\documentclass{scrartcl}

\usepackage{lmodern}
\usepackage[utf8]{inputenc}
\usepackage[T1]{fontenc}
\usepackage[english]{babel,csquotes}

\usepackage{amsmath,amssymb,amsthm}
\usepackage{mathtools,thmtools,stmaryrd,esint}
\usepackage{mathrsfs}

\usepackage[backend=biber,style=alphabetic,url=false,giveninits=true,sorting=anyt]{biblatex} 
\usepackage[colorlinks=true]{hyperref}

\usepackage[svgnames]{xcolor}
\usepackage{graphicx}
\usepackage{todonotes}
\usepackage{longtable,tabu}
\usepackage[shortlabels]{enumitem}
\usepackage{bm} 
\renewcommand{\textbf}[1]{\begingroup\bfseries\mathversion{bold}#1\endgroup}

\newcommand{\N}{\mathbb{N}}
\newcommand{\Z}{\mathbb{Z}}

\newcommand{\R}{\mathbb{R}}

\newcommand{\sph}{\mathbb{S}}

\newcommand{\eps}{\varepsilon}

\newcommand{\curl}{\mathrm{curl}}

\DeclareMathOperator{\dist}{dist}
\DeclareMathOperator{\argmin}{argmin}

\numberwithin{equation}{section}

\declaretheorem[name=Theorem,within=section]{thm}
\declaretheorem[name=Lemma,numberlike=thm]{lemma}
\declaretheorem[name=Proposition,numberlike=thm]{proposition}

\declaretheorem[name=Definition,numberlike=thm,style=definition]{dfn}

\declaretheorem[name=Question,numberlike=thm,style=remark]{question}

\declaretheorem[name=Remark,numberlike=thm,style=remark]{remark}

\allowdisplaybreaks\title{Clearing-out of dipoles  for minimisers of 2-dimensional discrete energies with topological singularities}

\date\today
\begin{document}

\author{Adriana Garroni\thanks{Sapienza Universit\`{a} di Roma. P.le A. Moro 5 00185 Roma, Italy. Email: \texttt{garroni@mat.uniroma1.it}}\quad Mircea Petrache\thanks{Pontificia Cat\'olica Universidad de Chile, Av. V.Mackenna 4860, Macul, Santiago, 6904441, Chile. Email: \texttt{mpetrache@uc.cl}}\quad Emanuele Spadaro\thanks{Sapienza Universit\`{a} di Roma. P.le A. Moro 5 00185 Roma, Italy. Email: \texttt{spadaro@mat.uniroma1.it}}}

\maketitle
\begin{abstract}
    A key question in the analysis of discrete models for material defects, such as vortices in spin systems and superconductors or isolated dislocations in metals, is whether information on boundary energy for a domain can be sufficient for controlling the number of defects in the interior. We present a general combinatorial dipole-removal argument for a large class of discrete models including XY systems and screw dislocation models, allowing to prove sharp conditions under which controlled flux and boundary energy guarantee to have minimizers with zero or one charges in the interior. The argument uses the max-flow min-cut theorem in combination with an ad-hoc duality for planar graphs, and is robust with respect to changes of the function defining the interaction energies. 
\end{abstract}
\tableofcontents
\medskip

\section{Introduction}

In this paper we use a general approach based on graph theory and combinatorial arguments in order to address a key question in the analysis of discrete models for material defects such as vortices in spin systems and superconductors or isolated dislocations in metals.

We consider the minimization of a class of discrete energies in dimension $2$ (which include discrete energies for screw dislocations and $XY$ models for spin systems) with a boundary condition. We determine topological conditions on the boundary datum that guarantee that there exist solutions that either do not exhibit singularities, or  have exactly one singularity in the domain.

Our motivating application is the variational problem for the characterisation of the core energy associated to a vortex configuration for the $XY$ model (or equivalently of a screw dislocation for a discrete model for anti-plane plasticity). For concreteness, we consider the following example. We take $\Omega=(-1,1)\times (-1,1)$ and $\eps=\frac1n$. We aim to show that the following problem
\begin{equation}\label{eq-intro-spin}
\min\left\{
\sum_{i,j\in\eps\Z^2\cap\Omega, |i-j|=\eps} |v(i)-v(j)|^2\, :\ v(i) = \frac{i}{|i|} \ \forall i\in \partial\Omega\cap \eps\Z^2\right\},
\end{equation}
minimized over all functions $v:\eps\Z^2\cap\Omega\to \sph^1$, \emph{has minimisers with only one vortex inside $\Omega$}.

\begin{figure}[h]
\centerline{{\def\svgwidth{100pt}
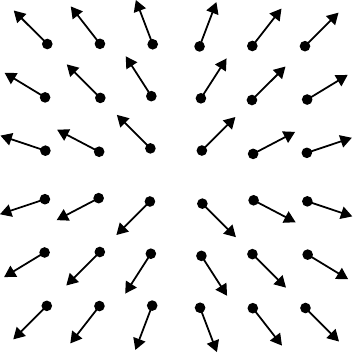}}
\caption{Schematic representation of a vortex for spins}
\label{vortex}
\end{figure} 

The energy in \eqref{eq-intro-spin} is known as the $XY$ model for spin systems, and the corresponding energy is denoted with ${\mathcal {XY}}_\eps$. 
While minimisers for the $\mathcal{ XY}_\eps$ energy (as well as for the corresponding energy $\mathcal{SD}_\eps$ for screw dislocations) are known to converge as $\eps\to 0$ to the function $u_0$ with one singularity at zero (see \cite{AliCic} and \cite{AlDLGarPo}),
the structure of the solution at $\eps>0$ was not known. In principle this asymptotic result does not exclude that the minimisers for $\eps>0$ exhibit short dipoles which disappear in the limit.
In this paper we show that this is not the case. Under appropriate conditions on the domain $\Omega$ and on the boundary datum, we prove that for $\eps>0$ there are minimisers without any dipoles. This information is used in a crucial way, for instance, in \cite{BachCicGarOrl} in order to relate the core energy of models for edge dislocations and partial dislocations with the one obtained with the XY model.


\medskip

The continuum counterpart of the above mentioned $\mathcal{XY}_\eps$ and $\mathcal{SD}_\eps$ models are the Ginzburg Landau energies for superconductivity
\begin{equation}\label{eq-GL}
F_\eps^{GL}(v)=\int_\Omega |\nabla v|^2dx+\frac1{\eps^2} \int_\Omega (|v|^2-1)^2dx,
\quad v:\Omega\subset\R^2\to \R^2, 
\end{equation}
and, in the context of antiplane elasticity, the following energy for screw dislocations 
\begin{equation}\label{eq-screw}
F_\eps^{screw}(\beta,\mu)=\int_{\Omega\setminus \cup_i B_\eps(x_i)}|\beta|^2 + C|\mu|,
\end{equation}
where $\beta:\Omega\to \R^2$ is the elastic strain and satisfies $\curl\beta=\sum_i d_i\delta_{x_i}:=\mu$. The measure $\mu$ represents the distribution of dislocations which are seen as topological singularities of the strain $\beta$. 

\medskip

These models are very much related. It has been proved that, at the leading orders,  they are equivalent in terms of $\Gamma$-convergence.
Interestingly, this equivalence also extends to the discrete models that are at the origin of these two semi-discrete models, where the parameter $\eps$ simulate the lattice spacing, as shown in \cite{AlCicPo} and \cite{AlDLGarPo} (in Section 2 these discrete models will be discussed in detail). 

The question of whether the minimisers of the energies in \eqref{eq-GL} (and \eqref{eq-screw}), for $\eps>0$, subject the boundary condition $\frac{x}{|x|}$, have only one vortex was raised first in \cite{BBH} and has been solved in various cases, but in full generality is still open (see \cite{Ignat_et_al_23} for a nice overview of the known results).

Here we show this result in the discrete setting, giving a very robust and constructive way of clearing out possible dipoles, without increasing the energy, and therefore exhibiting a minimiser with only one singularity. Our result applies to a quite general class of energies and boundary conditions. Moreover, it is a purely energetic argument that does not rely on the symmetry of the lattice nor of the boundary condition. Indeed, we show our result in the more general context of planar graphs, using a combinatorial argument based on the min-cut/max-flow theorem. Similar arguments have been first introduced in the study of singular radial connection for the Yang-Mills energy by the second author \cite{petrache2015singular}.
By a suitable reformulation of the problem on the dual graph, the topological singularities can be regarded as sinks or wells of a flow, and the argument consists in modifying the flow without increasing the energy, without changing the boundary condition and removing the dipoles.
We show by examples that the hypotheses used in this process are optimal and cannot be relaxed.

\section{Discrete models for material defects}\label{material-defects}

Here we introduce the discrete energies of interest and we state the main results. We will follow the approach of \cite{AO}; specifically, we will use the formalism and the notation in \cite{AlCicPo} and \cite{AlDLGarPo}  (see also \cite{Po}).
We consider a Bravais lattice in $\R^2$,  $\mathcal L=\{z_1 a_1+z_2 a_2 \,:\ z=(z_1,z_2)\in \Z^2\}$, where $a_1, a_2\in \mathbb R^2$ are fixed and form a basis. Cases of interest include the triangular lattice and the square lattice: without loss of generality, we will use the latter, i.e., ${\mathcal{L}}=\Z^2$, the other cases being recovered by a change of variable.
%
%
 
Let $\Omega\subset\R^2$ be a bounded open set. For every $\eps>0$ we define $\Omega_\eps\subset \Omega$ as follows
\begin{equation}
    \label{eq-dominio-eps}
    \Omega_\eps:=\bigcup_{i\in\eps\Z^2:\,i+\eps Q\subset \overline\Omega}(i+\eps Q), 
\end{equation}
where $Q=\left[0,1\right]^2$ is the unit square. Moreover, we set $\Omega_{\eps}^0:=\eps\Z^2\cap\Omega_\eps$, and the set of bonds
$$
\Omega_{\eps}^1:=\left\{(i,j)\in\Omega_{\eps}^0\times\Omega_{\eps}^0:\,\,|i-j|=\eps\right\}
$$ 

Finally, we define the \emph{discrete boundary} of $\Omega$ as
\begin{equation}\label{discrbdry}
\partial_\eps\Omega:=\partial\Omega_{\eps}\cap\eps\Z^2,
\end{equation}
and we denote by
$
\partial^+_\eps\Omega$
 the set of \emph{boundary bonds positively (counter-clockwise) oriented} (i.e., bonds $(i,j)\in \Omega_\eps^1$ and such that $(j-i)^\perp$, with notation $(x_1,x_2)^\perp:=(x_2,-x_1)$, points outside the domain $\Omega_\eps$).
Note that $\partial_\eps\Omega\subset\Omega_\eps^0$.

\begin{figure}[h]
\hskip4cm{\def\svgwidth{300pt}
\begingroup%
  \makeatletter%
  \providecommand\color[2][]{%
    \errmessage{(Inkscape) Color is used for the text in Inkscape, but the package 'color.sty' is not loaded}%
    \renewcommand\color[2][]{}%
  }%
  \providecommand\transparent[1]{%
    \errmessage{(Inkscape) Transparency is used (non-zero) for the text in Inkscape, but the package 'transparent.sty' is not loaded}%
    \renewcommand\transparent[1]{}%
  }%
  \providecommand\rotatebox[2]{#2}%
  \newcommand*\fsize{\dimexpr\f@size pt\relax}%
  \newcommand*\lineheight[1]{\fontsize{\fsize}{#1\fsize}\selectfont}%
  \ifx\svgwidth\undefined%
    \setlength{\unitlength}{506.46967615bp}%
    \ifx\svgscale\undefined%
      \relax%
    \else%
      \setlength{\unitlength}{\unitlength * \real{\svgscale}}%
    \fi%
  \else%
    \setlength{\unitlength}{\svgwidth}%
  \fi%
  \global\let\svgwidth\undefined%
  \global\let\svgscale\undefined%
  \makeatother%
  \begin{picture}(1,0.52795773)%
    \lineheight{1}%
    \setlength\tabcolsep{0pt}%
    \put(0,0){\includegraphics[width=\unitlength,page=1]{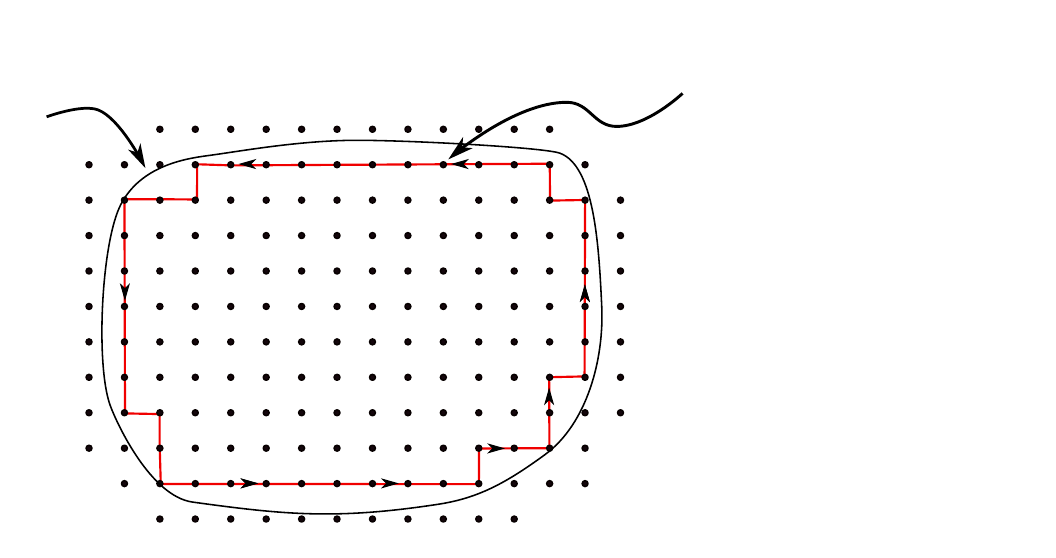}}%
    \put(0.63885349,0.45287799){\makebox(0,0)[lt]{\lineheight{1.25}\smash{\begin{tabular}[t]{l}$C=\partial_\varepsilon^+\Omega$\end{tabular}}}}%
    \put(-0.00149433,0.39401194){\makebox(0,0)[lt]{\lineheight{1.25}\smash{\begin{tabular}[t]{l}$\Omega$\end{tabular}}}}%
    \put(0,0){\includegraphics[width=\unitlength,page=2]{dominio-discreto2.pdf}}%
    \put(0.65807753,0.24497801){\makebox(0,0)[lt]{\lineheight{1.25}\smash{\begin{tabular}[t]{l}$i_0$\end{tabular}}}}%
  \end{picture}%
\endgroup%
}
\caption{A discrete domain (portion of a square lattice) and its discrete boundary.}
\label{bordo-discreto2}
\end{figure}

\subsection{The $\mathcal{SD}$-energy}\label{sec-SD}
The discrete screw dislocation energy, introduced in \cite{AlCicPo} and analysed in details in \cite{AlDLGarPo}, is defined on scalar functions $u:\Omega_\eps^0\to\R$ which represent the scaled vertical displacement of atoms (the actual vertical displacement would be $\eps u$). The energy accounts for crystal invariance and reads in its simplest form as
\[
   \mathcal{SD}_\eps(u):=\sum_{(i,j)\in \Omega_\eps^1}\mathsf{dist}^2(u(i)-u(j),\mathbb Z).
\]
We define the \emph{differential} $du$ of $u$ as a function $du: \Omega_\eps^1\to \mathbb R$ given by
\begin{equation}\label{du}
du(i,j):=u(j)-u(i).
\end{equation}
Note that by definition $du(i,j)=-du(j,i)$ for all $(i,j)\in \Omega_\eps^1$.
Now we define 
\begin{equation}\label{eq-proiezione}
\pi(y):= y-\argmin\{|y-z|:\ z\in \Z\}\in [-1/2,1/2]   
\end{equation}
where, when the minimiser is not unique, we choose the one closest to $0$. In particular this gives that
\[
\pi\big(du(i,j)\big)=-\pi\big(-du(i,j)\big),
\]
and hence $\pi(du(i,j))=- \pi(du(j,i))$. We also observe that in general
\[
\pi\big(du (i,j)\big) =\pi(u(j)-u(i)) \neq \pi(u(j))- \pi(u(i)),
\]
and
$$
|\pi(y)|= \min\{|y-z|:\ z\in \Z\}= \dist(y,\Z).
$$
Therefore, the energy $\mathcal{SD}_\eps$ can equivalently be written as
\[
\mathcal{SD}_\eps(u)=\sum_{(i,j)\in \Omega_\eps^1}|\pi(du(i,j))|^2.
\]
\begin{figure}[h]
\centerline{\includegraphics[scale=.3]{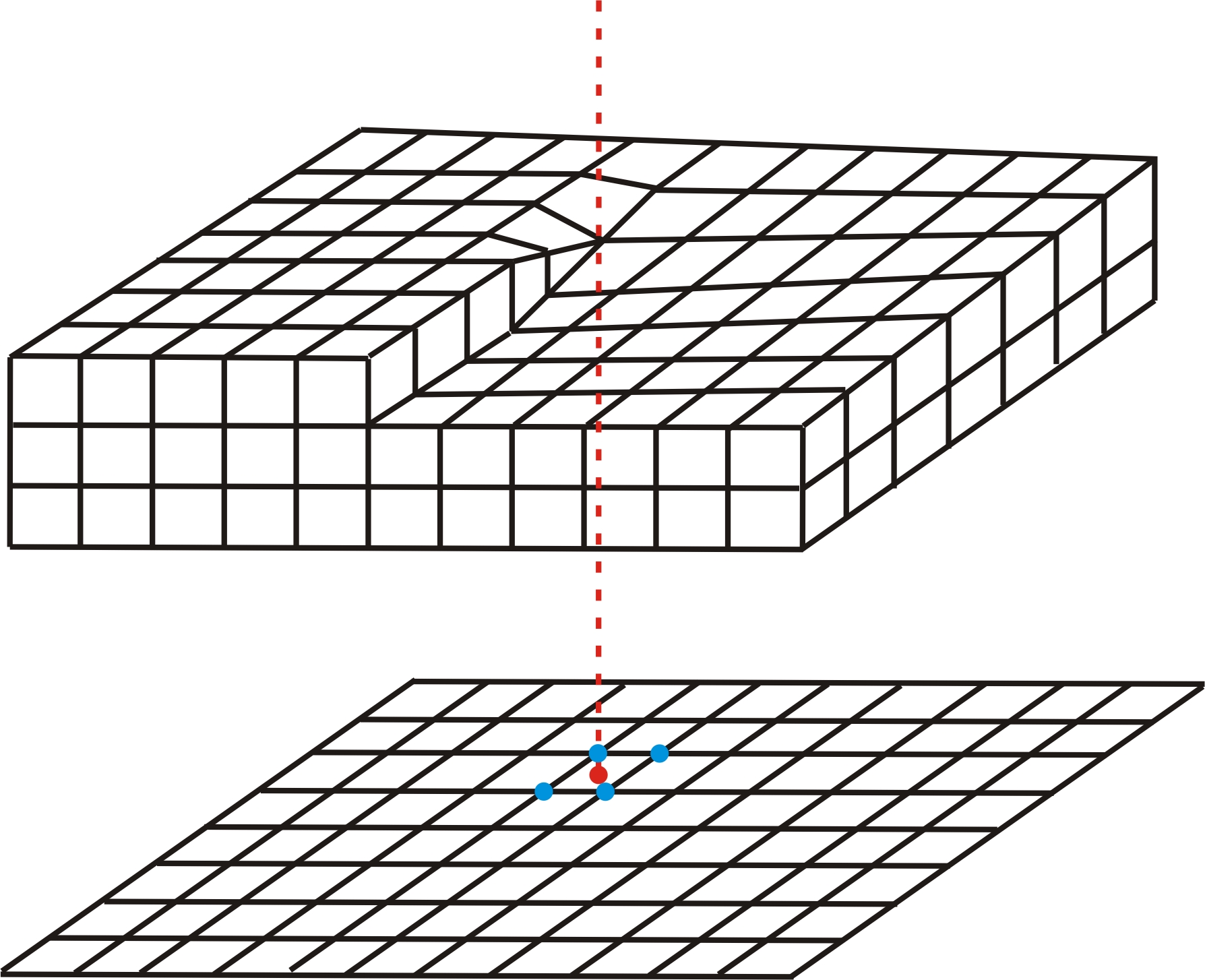}}
\caption{Schematic representation of a screw dislocation in a cubic lattice}
\label{fig-screw}
\end{figure}

Here we introduce the notion of discrete topological singularities. Given a function $u:\Omega_\eps^0\to \R$ we define the discrete 
circulation of $\pi(du)$ on a cell $i+\eps Q$ as
\begin{equation}\label{disctop}
\begin{split}
\alpha_u(i):= \pi (du)(i,i+\eps e_1) 
+   
\pi (du)(i+\eps e_1,i+\eps e_1+\eps e_2)\\
- \pi (du)(i+\eps e_2, i+\eps e_1+\eps e_2) - \pi (du)(i,i+\eps e_2) .
\end{split}
\end{equation}
Notice that, by definition, for every $(i,j)\in \Omega_\eps^1$ there is $z(i,j)\in \Z$ such that
\[
\pi(du(i,j))= du(i,j)+z(i,j).
\]
Therefore, we get immediately that $\alpha_u(i)\in\Z$.
We can actually show that $\alpha_u(i)\in \{-1,0,1\}$. Indeed, since $\pi(du)\in [-\frac12,\frac12]$, then $\alpha_u(i)\in \{-2,-1,0,1,2\}$. In order to exclude the values of $\pm 2$, first note that if $\alpha_u(i)=2$ then necessarily
\begin{gather}\label{eq-unmezzo}
\begin{split}
&\pi (du)(i,i+\eps e_1) 
=     \pi (du)(i+\eps e_1,i+\eps e_1+\eps e_2)=\\
&= - \pi (du)(i+\eps e_2, i+\eps e_1+\eps e_2) =
- \pi (du)(i,i+\eps e_2)=\frac12,
\end{split}
\end{gather}
the case with circulation $-2$ being analogous. Next, by the very definition of $\pi$, $\pi(du(i,j))=\frac12$ if and only if $du(i,j)=\frac12 + z(i,j)$ with $z(i,j)\in\N$ (in particular $z(i,j)$ non-negative). One can check that this property is not compatible with \eqref{eq-unmezzo}.

The circulation from \eqref{disctop} defines a discrete notion of $\curl$ for $\pi(du)$ around the face $i+\eps Q$, which in general may be non zero, and that can be interpreted,  
borrowing the terminology of the models for dislocations in crystals, as a screw dislocation of Burgers vector $b_i e_3$, $b_i=\alpha_u(i)\in \{-1,0,1\}\subset\Z$.  Similarly we can define the circulation of $\pi(du)$ along a closed circuit $C=\cup_{(i,j)\subset\partial_\eps D}(i,j)=(i_0,i_1)\cup(i_1,i_2)\cup\ldots\cup (i_n,i_0)$ which is the boundary of a discrete domain $D_\eps$, as
\begin{equation}
    \label{eq-circulation}
    \alpha_u(C)=\sum_{k=0}^{n-1}\pi(du)(i_k,i_{k+1}) +\pi(du)(i_n,i_0).
\end{equation}

\begin{remark}
    Note, for future reference, that a good way to understand the quantity $\alpha_u(i)$ is to write 
$$
v = v(u):= e^{2\pi i u}.
$$
The presence of a cell $i+\eps Q$ with non zero curl then corresponds to the presence of a vortex structure for $v$, and the quantity $\alpha_u(i)$ represents a notion of discrete vorticity for the $\sph^1$ valued field $v$.
\end{remark}

 Finally, we define the vorticity measure (or dislocation density)  $\mu(u)$ as follows
\begin{equation}\label{defvor}
\mu(u)=\mu_u:= \sum_{i\in\Omega_\eps^0}\alpha_u(i) \delta_{i+ \frac{\eps}{2} (e_1 + e_2)}.
\end{equation}
This definition of vorticity extends to $\mathcal \sph^1$ valued fields in the obvious way, by setting, with an abuse of notation, $\mu_v := \mu_u$ where $u$ is any function  such that $v(u) = v$. 

With this, from \eqref{eq-circulation} we have
\begin{equation}
    \alpha_u(\partial_\eps^+ D)= \mu_u(D_\eps),
\end{equation}
where $\partial_\eps^+ D$ is oriented counterclockwise. In this sense $\mu(u)$ represents the discrete $\curl$ of the map $\pi(du)$.

The problem that we address here is the following: we fix $u_0:\Omega_\eps^0\to \R$ and we consider the minimization problem
\begin{equation}\label{eq-min-SD}
    \min\big\{\mathcal{SD}_\eps(u) : \ u:\Omega_\eps^0\to \R,\ u(i)=u_0(i)\ \forall \ i\in \partial_\eps\Omega \big\}.
\end{equation}

Our main question of interest is then which type of topological conditions on the boundary datum $u_0$ guarantee the existence of minimisers free of defects or having only one singularity in $\Omega$.

\subsection{The ${XY}$-energy}\label{sec-XY}
The model presented above is equivalent at leading order to the well known XY model for spin systems in statistical mechanics (see for instance \cite{AlCic} and the references therein).
For any $v:\Omega_\eps^0\to\mathbb{S}^1$, we define
\begin{equation}\label{xy1}
\mathcal{XY}_\eps(v):=\frac{1}{2}\sum_{(i,j)\in\Omega_{\eps}^1}|v(i)-v(j)|^2.
\end{equation}
We can rewrite this energy in terms of the phase $u$ of $v$:
\begin{equation}\label{XYvsSD}
\mathcal{XY}_\eps(v) = \sum_{(i,j)\in\Omega_{\eps}^1}\big(1-\cos(2\pi (u(i)-u(j)) \big)\qquad \text{ with } \ v= e^{2\pi i u},
\end{equation}
which we notice can be also written as 
\begin{equation}\label{XYf}
\mathcal{XY}_\eps(v) 
= \sum_{(i,j)\in\Omega_{\eps}^1}f\big(\dist(u(i)-u(j), \Z)\big)
=\sum_{(i,j)\in\Omega_{\eps}^1}f\big(|\pi(du(i,j))|\big)\qquad \text{ with } \ v= e^{2\pi i u}
\end{equation}
where $f(t)=1-\cos(2\pi t)$, a function that we note to be increasing for $t\in[0,1/2]$.
We notice also that  
\begin{equation}\label{FmaggXY}
\mathcal{SD}_\eps(u)\ge \frac1{2\pi^2} \, \mathcal{XY}_{\eps}(e^{2\pi i u}).
\end{equation}
In the energy regime in which singularities emerge (i.e., for energy of the order $|\log\eps|$), the $\mathcal{XY}_\eps$ and $\mathcal{SD}_\eps$ models have the same asymptotic expansion in terms of $\Gamma$-converges (see \cite{AlCicPo} and \cite{AlDLGarPo}).

\begin{remark}\label{zerozero}

Another very classical (continuum) model which can be proved to be equivalent (in certain regimes) to $\mathcal{XY}_\eps$ and $\mathcal{SD}_\eps$ is the Ginzburg Landau model for superconductivity.  
In particular one can see that the notion of vortices introduced above has a very natural continuum (topological) counterpart.

To clarify this, set  $\left\{T_i^\pm \right\}$ to be the family of the $\eps$-simplices of $\R^2$ for $i\in\eps\Z^2$, where $T_i^+$ has vertices $\{i, i+ \eps e_1, i +\eps e_2\}$ and $T_i^-$ has vertices $\{i, i-\eps e_1, i-\eps e_2\}$.
For any $v:\Omega_{\eps}^0\to\sph^1$, we denote by $\tilde v:\Omega_\eps\to\R^2$ the piecewise affine interpolation of $v$ on the triangulation $\left\{T^\pm_i \right\}$.
It is easy to see that, up to boundary terms,  $\mathcal{XY}_\eps(v)$  corresponds to the Dirichlet energy of $\tilde v$ in $\Omega_\eps$; more precisely
\begin{equation}\label{disccont}
\frac{1}{2}\int_{\Omega_{\eps}}|\nabla\tilde{v}|^2 d x+\frac{1}{2}\int_{(\partial\Omega_\eps)_\eps}|\nabla\tilde v|^2dx\ge \mathcal{XY}_{\eps}(v)\ge\frac{1}{2}\int_{\Omega_{\eps}}|\nabla\tilde{v}|^2 d x,
\end{equation}
where  $(\partial\Omega_\eps)_\eps=\{i+\eps Q: i\in \partial_\eps\Omega\}$. 
Moreover, one can easily verify that  if $|\tilde v|> c>0$ on $\partial  A_\eps$, for some open set $A\subset\Omega$, then
\begin{equation}\label{degdisccont}
\mu_v(A_\eps)=\deg(\tilde v,\partial A_\eps),
\end{equation}
where the \emph{degree} of a function $w\in H^{\frac{1}{2}}(\partial A;\R^2)$, with $|w|\ge c>0$, is defined by
\begin{equation}\label{deg}
\deg(w,\partial A):=\frac{1}{2\pi}\int_{\partial A}\left(\frac{w}{|w|}\times \nabla\frac{w}{|w|}\right)\cdot\tau\ ds\,,
\end{equation}
with  $v\times \nabla v:= v_1\nabla v_2 - v_2\nabla v_1$ for $v\in H^1(A; \R^2)$.
In particular, whenever $|\tilde v| >0$ on $i+\eps Q$ we have $\mu_v(i+\eps Q) = 0$.  
\end{remark}

\subsection{Singularities of the minimisers}\label{ssec:exampleball}
The questions raised in Sections \ref{sec-SD} and \ref{sec-XY} can be rephrased in a framework which includes and generalises the two settings. We consider the discrete energy
\begin{equation}
    \label{eq-energy-general}
    E_\eps(u) = \sum_{(i,j)\in \Omega_\eps^1}g(u(i)-u(j)),
\end{equation}
where $u:\Omega_\eps\to\R$ and $g:\R\to[0,+\infty)$ satisfies
\begin{equation}\label{eq-energy-gen-f}
    g(t)= f(\dist(t,\Z)),
\end{equation}
with $f:[0,1/2]\to \R$ non-descreasing. 

We denote by ${\mathcal A}_\eps(\Omega):=\{u:\Omega_\eps\to \R\}$.
Here we assume that $\Omega$ is an open simply connected domain. Then, we consider a function $u_0\in {\mathcal A}_\eps(\Omega)$ and the following minimimum problem
\begin{equation}\label{eq-min-general}
    \min\big\{E_\eps(u) : \ u\in {\mathcal A}_\eps(\Omega)\ , \  u(i)=u_0(i)\ \forall \ i\in \partial_\eps\Omega \big\}.
\end{equation}

We will make two alternative assumptions on the boundary condition. Let 
$$
\partial^{+}_\eps\Omega:=C=(i_0,i_1)\cup (i_1,i_2)\cup\ldots\cup (i_n,i_0),
$$  
with $(i_k,i_{k+1}), (i_n,i_0)\in \Omega^1_\eps$ being the discrete boundary oriented counter-clockwise:
\begin{itemize}
  \item[{(H0)}] $\pi(du_0(i,j))=du_0(i,j)$ for all bonds $(i,j)\in C$ except at most one;
    \item[{(H1)}] $\alpha_{u_0}(C)=0$ and $\sum_{k=0}^{n-1} |\pi(du_0)(i_k,i_{k+1})| + |\pi(du_0)(i_n,i_0)|\leq 1$;
    \item[{(H2)}] $|\alpha_{u_0}(C)|=1$ and $\sum_{k=0}^{n-1} |\pi(du_0)(i_k,i_{k+1})| + |\pi(du_0)(i_n,i_0)|= 1$.
\end{itemize}

\begin{remark}
    \label{rem-genaral-boundary-data}
Assumption (H0) is satisfied by a very large class of boundary conditions. For example, for any continuous $v_0:\overline\Omega\to \mathbb{S}^1$ 
we can construct a lifting $u_0$ of $v_0$ in a neighbourhood of $\partial \Omega$, i.e., $v_0(x)= e^{i u_0(x)}$,
which is continuous at all points of $\partial_\eps\Omega$ except $i_0$ and satisfies (H0) on $\partial_\eps\Omega$ for $\eps$ small enough. Precisely, if $\omega(t)$ is the modulus of continuity of $v_0$ on $\partial_\eps\Omega$ ($\omega$ increasing with $\omega(t)\to 0$ as $t\to 0$), in order to have (H0) be satisfied it is enough to choose $\eps\leq \omega^{-1}(\frac1{2\sqrt2 })$. Indeed, given $(i_k,i_{k+1})\in C$, with   $|i_k-i_{k+1}|=\eps$ we have
    $$
    |u_0(i_{k+1})-u_0(i_k)|\leq \sqrt2|v_0(i_{k+1})-v_0(i_k)|\leq \sqrt2 \omega(\eps)
    $$
so that if $\eps\leq \omega^{-1}(\frac1{2\sqrt2})$, then $|u_0(i_{k+1})-u_0(i_k)|\leq \frac12$ for all $k=\{0,\ldots, n-1\}$, and therefore it satisfies (H0).
\end{remark}

Our main result is the following.


\begin{thm}\label{th-main-general}
Let $u_0:\partial_\eps\Omega\to \R$ satisfy assumption (H0).
\begin{enumerate}
\item If $u_0$ satisfies in addition (H1), then there exists a minimiser $u$ of $E_\eps$ with no singularities in $\Omega_\eps$, i.e., with $\mu_u=0$. Furthermore, if in (H1) the strict inequality holds, then all minimizers have no singularites in $\Omega_\eps$.
\item If $u_0$ satisfies (H2), then there exists a minimiser $u$ $E_\eps$ with exactly one singularity in $\Omega_\eps$, i.e., 
$$\mu_u=\mu_{u_0}(\Omega_\eps)\delta_{x_0}$$ 
for some $x_0\in\Omega_\eps^0+\frac{\eps}{2}(e_1+e_2)$.
\end{enumerate}
\end{thm}

The above theorem clearly applies to both the $\mathcal{XY}_\eps$ energy and $\mathcal{SD}_\eps$ energy.
The proof of this theorem will be a consequence of a more general result (see Theorem \ref{thm:planargraphSD}) in the context of  bi-directed graphs $(E,V)$ applied to the particular case $V=\Omega_\eps^0$ and $E=\Omega_\eps^1$.

\begin{remark}
It is not guaranteed, and in fact it can be false, that \emph{all} minimisers of $E_\eps$ are free of singularities if (H1) holds, or that all minimisers have only one singularity if (H2) is verified. Figure \ref{Counterexample} illustrates a case of a boundary condition satisfying conditions (H0) and (H1) and two minima for the $\mathcal{SD}_\epsilon$-energy (i.e., for $f(t)=\frac12 t^2$ in notation \eqref{eq-energy-gen-f}) with equal boundary datum, one with no vortices in the inner squares (left image) and another with two unit charges of opposite signs in the two top squares (right image).

\begin{figure}[h]
\centerline{{\def\svgwidth{220pt}
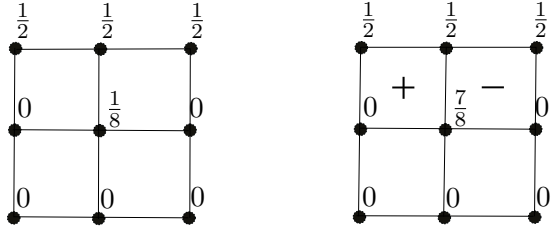}}
\caption{Two minima of problem \eqref{eq-min-SD} for a $2\times2$ square lattice with a boundary conditions satisfying (H0) and (H1). We indicate a charge of $+1$ and $-1$ by a $+$ and a $-$ sign at the centers of corresponding faces. Empty faces have zero charge.}
\label{Counterexample}
\end{figure}  
\end{remark}

\subsubsection{The one-vortex solution in star-shaped domains}
Explicit examples can be given in star shaped domains $\Omega$.
We assume that for a given $\eps>0$ also the discrete domain $\Omega_\eps$ is star shaped; without loss of generality we assume $0\in \Omega$ and that $\Omega$ and $\Omega_\eps$ are star shaped with respect to $0$. We consider a boundary datum $u_0$ of the form
\[
    u_0\big(r\sin(\theta), r\cos(\theta)\big):=\psi\left(\frac{\theta}{2\pi}\right)\text{ for }r>0,\quad u_0(0):=0,
\]
for some non-decreasing $\psi:[0,1]\to [0,1]$. We are interested in minimizing $\mathcal{SD}_\eps(u)$ over the set of maps
\[
    u:\Omega_\eps\to \R,\quad u(i)=u_0(i) \quad \text{ for }i\in \partial_\eps \Omega.
\]
Let $C=(i_0,i_1)\cup (i_1,i_2)\cup\ldots\cup (i_n,i_0)$ (with $i_0$ on the positive $x_1$ axis, as in Figure \ref{bordo-discreto2}) be the discrete boundary $\partial^{+}_\eps \Omega$ oriented counter-clockwise. 
Assume that $u_0(i_{k+1})-u_0(i_k) \leq 1/2$ for all $k\in\{0,\ldots, n-1\}$ (as explained in Remark \ref{rem-genaral-boundary-data}, this is guaranteed for instance if $\psi$ is $L$-Lipschitz and $\eps<\frac{\pi\dist(0,\partial\Omega)}{L}$). Then, $u_0$ satisfies assumptions (H0) and (H2) on $\partial^{+}_\eps\Omega$. Indeed  one can immediately check that if  $k\in\{0,\ldots, n-1\}$ we have $$|\pi(du_0)(i_k,i_{k+1})|=du_0(i_k,i_{k+1})=u_0(i_{k+1})-u_0(i_k),$$ while  $|\pi(du_0)(i_n,i_0)|= 1+ (u_0(i_0)-u_0(i_n))$, so that 
\begin{equation}\label{eq-condizioneh2-esempio}
    \sum_{k=0}^{n-1}|\pi(du_0)(i_k,i_{k+1})| + |\pi(du_0)(i_n,i_0)|=1.
\end{equation}
Therefore as a consequence of Theorem \ref{th-main-general} we get that there exists a minimiser with only one singularity in $\Omega_\eps$. The same conclusion holds for the minimiser of the  ${\mathcal {XY}}_\eps$ energy with the typical degree $1$ boundary condition $v_0(x)=e^{u_0(x)}$ on $\partial_\eps\Omega$.
For example,  if $\Omega=B_R$ and $v_0(x)=\frac{x}{|x|}$ we obtain that there is a minimiser of the $\mathcal{XY}_\eps$  with exactly one vortex in $(B_R)_\eps$. We believe that if we have a not too small number of interior lattice points, all minimizers should exhibit at most one vortex. However the approach from this work is not sufficient to show that removing the dipoles \emph{strictly decreases} the energy in this case.


\medskip

The rest of the paper is devoted to the reformulation of the above questions in more general terms using bi-directional graphs $G=(V, E)$.

\section{The Graph formulation}

Our proof strategy for the general version of Theorem \ref{th-main-general} is to first prove a similar result on $1$-forms with integral divergence on general graphs (see Section \ref{sec-1fid}). In this section we first state in Theorem \ref{thm:planargraphSD} a version of our main theorem on a suitable class of admissible planar graphs, we then translate it in Lemma \ref{lem:dual} for dual graphs. We prove the corresponding statements in Section \ref{sec-1fid}.
%
%
%
\subsection{Planar complexes and formulation of the main result}
A {\bfseries bidirectional} graph $G=(V,E)$, is a graph with edges satisfying the property
\begin{equation}\label{e.bidirectional}
(a,b)\in E\Leftrightarrow (b,a)\in E.     
\end{equation}
A finite graph $G=(V,E)$ will be called \textbf{planar} when the vertices $V$ are bijectively identified with a finite set of points in the plane, i.e., $V\subset \mathbb R^2$, and the edges $e=(a,b)\in E$ are identified with oriented curves in $\mathbb R^2$ containing as endpoints the vertices $a,b\in V$, with the condition that such curves intersect only at their endpoints and this happens if and only if the corresponding edges share the corresponding vertex. If $a,b\in V$ and $(a,b)\in E$ we also write $a\sim b$. 
By abuse of notation we denote the realizing curve corresponding to edge $(a,b)\in E$ by $(a,b)$. Furthermore, we require that whenever $(a,b), (b,a)\in E$, then $-e:=(b,a)$ is identified with the curve $(a,b)$ with opposite orientation. 
\begin{figure}[h]
\centerline{{\def\svgwidth{200pt}
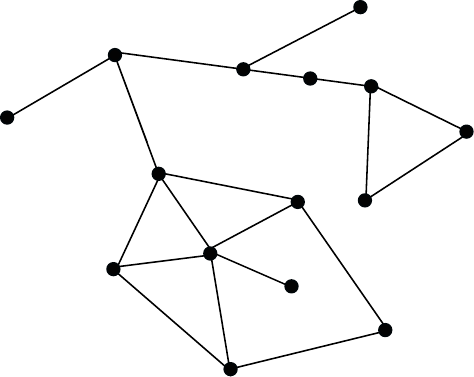}}
\caption{Example of a planar graph}
\label{planar graph}
\end{figure}  
If $G=(V,E)$ is a planar graph, then its \textbf{faces} are the bounded connected components of $\mathbb R^2\setminus \bigcup_{e\in E}e$. We reserve the letter $F$ for denoting the set of faces of a planar graph, and we call $(V,E,F)$ the associated \textbf{planar complex}. We define the \textbf{exterior of $G$}, denoted $f_{\mathrm{out}}$, to be the unbounded connected component of $\mathbb R^2\setminus \bigcup_{e\in E}e$. 

For $f\in F\cup\{f_{\mathrm{out}}\}$ we denote by $\partial f$ the \textbf{oriented boundary complex} of $f$, i.e., the set of the edges $e\in E$ which cover the topological boundary of $f$ with standard counterclockwise orientation if $f\in F$, or clockwise orientation if $f=f_{\mathrm{out}}$. Here edges that are interior to $f$ appear twice in $\partial f$, with a copy of each orientation.

For a given planar graph $G=(V,E)$ with face set $F$, $e\in E$ is an \textbf{interior edge} if $e$ is in common between two faces $f_1,f_2\in F$. A \textbf{boundary edge} will be: (a) an edge of $\partial f_{\mathrm{out}}$;
or equivalently 
 (b) an edge $e\in E$ which belongs either to the boundary of a single face $f\in F$, or does not belong to the boundary of any face $f\in F$.
 The \textbf{boundary complex} $(V^\partial, E^\partial)$ of $G=(V,E,F)$ is defined as follows:
\[
    E^\partial :=\{-e:\ e\in \partial f_{\mathrm{out}}\},\quad V^\partial:=\{a,b: (a,b)\in E^\partial\}.
\]
Note that the boundary complex $(V^\partial, E^\partial)$ is an oriented but non-bidirected subgraph of $(V, E)$.

We say that a graph $G$ is \textbf{connected} if the union of its edges is a connected set.
A planar complex $(V,E,F)$ with boundary complex $(V^\partial, E^\partial)$ is \textbf{admissible} if the graph $(V,E)$ is connected and all edges of $E^\partial$ belong to $\partial f$ for at least one $f\in F$. 

%
Next, as in the previous section, for $u:V\to\mathbb R$ we define $du:E\to\mathbb R$ by $$du(i,j):=u(j)-u(i),$$ and we say that a map $\alpha:E\to\mathbb R$ is a \textbf{discrete $1$-form} if $\alpha(j,i)=-\alpha(i,j)$ whenever $(i,j)\in E$ (and then $(j,i)\in E$). Given a subsets $E'\subseteq E$ with a little abuse of notation we write
\begin{equation}\label{eq:alphaofedges}
    \alpha\left(E'\right):=\sum_{(a,b)\in E'}\alpha(a,b).
\end{equation}
\begin{dfn} If $(V,E,F)$ is a planar complex and $\alpha: E\to\mathbb R$ is a $1$-form, we define $\mathsf{curl}\,(\alpha):F\to\mathbb R$ as
\[
    \mathsf{curl}\,(\alpha)(f):=\alpha(\partial f).
\]
For a fixed planar complex $(V,E,F)$ we say that a discrete $1$-form $\alpha:E\to\mathbb R$ has \textbf{integer-valued curl} if $\mathsf{curl}\,(\alpha)(f)\in\mathbb Z$ for all $f\in F$.
\end{dfn}
Now we formulate our main hypotheses. Let $(V,E,F)$ be a planar $2$-dimensional complex with boundary complex $(V^\partial, E^\partial)$.
\medskip
Further, recall that $$\pi(y):=y-\mathsf{argmin}\{|y-z|:\ z\in\mathbb Z\},$$ where, if the minimum is achieved by more than one $z\in \mathbb Z$ (i.e., in the case $y\in \frac12 +\Z$), we set $\pi(y)=y-z$ for the value of $z$ closest to the origin (which gives $-\frac12$ if $y<0$ and $\frac12$ if $y>0$). The energies of interest for functions $u:V\to \mathbb R$ are increasing functions of $|\pi(du(i,j))|=\mathsf{dist}(du(i,j),\mathbb Z)$ for $(i,j)\in E$.
We consider the following hypotheses over maps $u:V\to \mathbb R$.
\begin{enumerate}
    \item[(h0)] The condition $\pi(du(e))=du(e)$ for edges $e=(i,j)\in E^\partial$ holds with the exception of at most one such edge.
    \item[(h1)] $\sum_{e\in E^\partial}\pi(du(e))=0$ and $\sum_{e\in E^\partial}|\pi(du(e))|\leq 1$.
    \item[(h2)] $\sum_{e\in E^\partial}\pi(du(e))\in\{\pm 1\}$ and $\sum_{e\in E^\partial}|\pi(du(e))|= 1$.
\end{enumerate}
\begin{thm}\label{thm:planargraphSD}
Let $(V,E)$ be an admissible planar graph.
Let $u:V\to \mathbb R$ be a function such that $\pi\circ du$ has integer-valued curl and such that (h0) holds. Then the following statements are true.
    \begin{enumerate}
        \item If $u$ satisfies (h1) then there exists a function $\widetilde u:V\to\mathbb R$ such that $\widetilde u = u$ over $V^\partial$, $|\pi(d\widetilde u)|\leq |\pi(du)|$ and $\mathsf{curl}\,(\pi\circ d\widetilde u) = 0$. Furthermore, if (h1) holds with strict inequality $\sum_{e\in E^\partial}|\pi(du(e))|<1$ and $\mathsf{curl}(\pi\circ du)\neq 0$, then $\widetilde u$ can be chosen so that strict inequality $|\pi(d\widetilde u(e))|<|\pi(d u(e))|$ holds on at least one edge $e\in E$.
        \item If $u$ satisfies (h2) then there exists a function $\widetilde u:V\to\mathbb R$ such that $\widetilde u = u$ over $V^\partial$, $|\pi(d\widetilde u)|\leq |\pi(du)|$ and a face $f_0\in F$ such that we have 
        \[
            \mathsf{curl}\,(\pi\circ d\widetilde u)=\left(\sum_{e\in E^\partial}\pi(du(e))\right)\delta_{f_0},
        \]
where
\[        
  \delta_{f_0}(f):=\left\{\begin{array}{ll}0& \text{ if }f\in F\setminus\{f_0\},\\ 1 &\text{ if }f=f_0.\end{array}\right.
 \]
\end{enumerate}
\end{thm}
\begin{remark}\label{rem-thm-lattices}
As an immediate consequence of the above theorem we deduce Theorem \ref{th-main-general}. Indeed, we apply Theorem \ref{thm:planargraphSD} with $V=\Omega_\eps^0$ and $E=\Omega_\eps^1$, noticing that, in view of  \ref{defvor} the assumptions (h0), (h1) and (h2) in Theorem \ref{th-main-general} coincide with the corresponding conditions in Theorem \ref{thm:planargraphSD}.
\end{remark}
\begin{remark}\label{rem-sharp-hyp}
The hypotheses of Theorem \ref{thm:planargraphSD} are sharp, as can be easily inferred by the counterexamples given the next section, Remark \ref{r:optimal}. In particular, we cannot hope to have an extension of the strong dipole removal result of Theorem \ref{thm:planargraphSD} for $\sum_{e\in E^\partial}|\pi(du(e))|>1$. However, in Appendix \ref{sec:openproblems} we discuss possible extensions of our result in this direction by weakening the hypotheses of the theorem.
\end{remark}
\begin{remark}\label{rem-removeadmiss}
 Theorem \ref{thm:planargraphSD} may be directly extended to connected planar non-admissible graphs. Indeed, removing  the edges that do not belong to the boundary of any face $f\in F$, we reduce to a finite union of planar admissible graphs. The boundary of these subgraphs are contained in $E^\partial$, therefore the boundary condition $\tilde u$ satisfies (h0) and Theorem \ref{thm:planargraphSD} can be applied to each of these graphs separately.
\end{remark}

\subsection{Passing to the dual graph}
For admissible planar graphs $(V,E,F)$ we define the dual graph and reformulate the hypotheses (h0)-(h2) of Theorem \ref{thm:planargraphSD} in the dual graph. Roughly speaking, we pass from the study of discrete vorticity $\mathsf{curl}\,(\alpha)$ for functions on graph vertices, to the study of the discrete divergence $\mathsf{div}\,(\gamma)$ for an associated $1$-form $\gamma$ on the edges of the dual graph. This reformulation allows to deploy the theory of maxflow-mincut duality for flows on the dual graph, which will be exploited in Section \ref{sec-1fid} to prove dual graph analogues of the statements of Theorem \ref{thm:planargraphSD}. In Section \ref{sec:dualityend} we then return to the initial planar complex and complete the proof of Theorem \ref{thm:planargraphSD}. 

Observe that boundary complex edges and vertices play a special role in the formulation of 
Theorem \ref{thm:planargraphSD}, and below we adapted the definition of duality so that it fits well with our proof strategy in Sections \ref{sec-1fid} and \ref{sec:dualityend}.
\begin{dfn}\label{def:dualgraph} 
    Let $(V,E,F)$ be an admissible planar complex, let $(V^\partial, E^\partial)$ be its boundary complex. Then the associated \textbf{(oriented) dual graph} $G^\perp=(V^\perp,E^\perp)$ is given by the following:
\begin{itemize} 
\item \textbf{Interior vertices} $V^{\perp,int}$ will be in bijection to faces $f\in F$, and \textbf{boundary vertices} $V^{\perp,\partial}$ will be in bijection with edges from $E^\partial$. More precisely, we denote by $f^\perp$ the vertex associated to face $f\in F$ and by $v_e$ the vertex associated to boundary edge $e\in E^\partial$. Then we set
\[
    V^{\perp,int}:=\{f^\perp:\ f\in F\},\quad V^{\perp,\partial}:=\{v_e:\ e\in E^{\partial}\},
\]
and we define $V^\perp:=V^{\perp, int}\cup V^{\perp,\partial}$.
\item Let $e=(i,j)\in E^{int}$ belong to the intersection of $\partial f_1$ and $\partial f_2$ for $f_1, f_2\in  F$, with vertex ordering $(i,j)$ fixed so that the unit normal vector field $\nu_{12}$ along $e$ pointing from $f_1$ to $f_2$ is the counterclockwise rotation by $\pi/2$ of $\dot e$. Then we set 
\[
    e^\perp:=(f_1^\perp, f_2^\perp).
\]
 In particular, this includes edges interior to a face, in the case $f_1=f_2$, in which case edge $e^\perp$ is a loop between $f_1^\perp$ and itself. 
\item If $e=(i,j)\in E^\partial$, note that directly from the definition of the boundary complex, there exists a unique face, denoted $f_e\in F$, for which $e$ parametrises part of $\partial f$ with counterclockwise orientation. Then we set 
\[
    e^\perp:=(f^\perp_e, v_e).
\]
\item Finally, we define 
\[
    E^\perp:=\{e^\perp:\ e\in E^{int}\}\cup \{(f_e^\perp, v_e), (v_e, f_e^\perp):\ e\in E^\partial\}.
\]
\end{itemize}
\end{dfn}
\begin{figure}[h]
\hskip4cm{\def\svgwidth{200pt}
\begingroup%
  \makeatletter%
  \providecommand\color[2][]{%
    \errmessage{(Inkscape) Color is used for the text in Inkscape, but the package 'color.sty' is not loaded}%
    \renewcommand\color[2][]{}%
  }%
  \providecommand\transparent[1]{%
    \errmessage{(Inkscape) Transparency is used (non-zero) for the text in Inkscape, but the package 'transparent.sty' is not loaded}%
    \renewcommand\transparent[1]{}%
  }%
  \providecommand\rotatebox[2]{#2}%
  \newcommand*\fsize{\dimexpr\f@size pt\relax}%
  \newcommand*\lineheight[1]{\fontsize{\fsize}{#1\fsize}\selectfont}%
  \ifx\svgwidth\undefined%
    \setlength{\unitlength}{489.24186646bp}%
    \ifx\svgscale\undefined%
      \relax%
    \else%
      \setlength{\unitlength}{\unitlength * \real{\svgscale}}%
    \fi%
  \else%
    \setlength{\unitlength}{\svgwidth}%
  \fi%
  \global\let\svgwidth\undefined%
  \global\let\svgscale\undefined%
  \makeatother%
  \begin{picture}(1,0.72053332)%
    \lineheight{1}%
    \setlength\tabcolsep{0pt}%
    \put(0,0){\includegraphics[width=\unitlength,page=1]{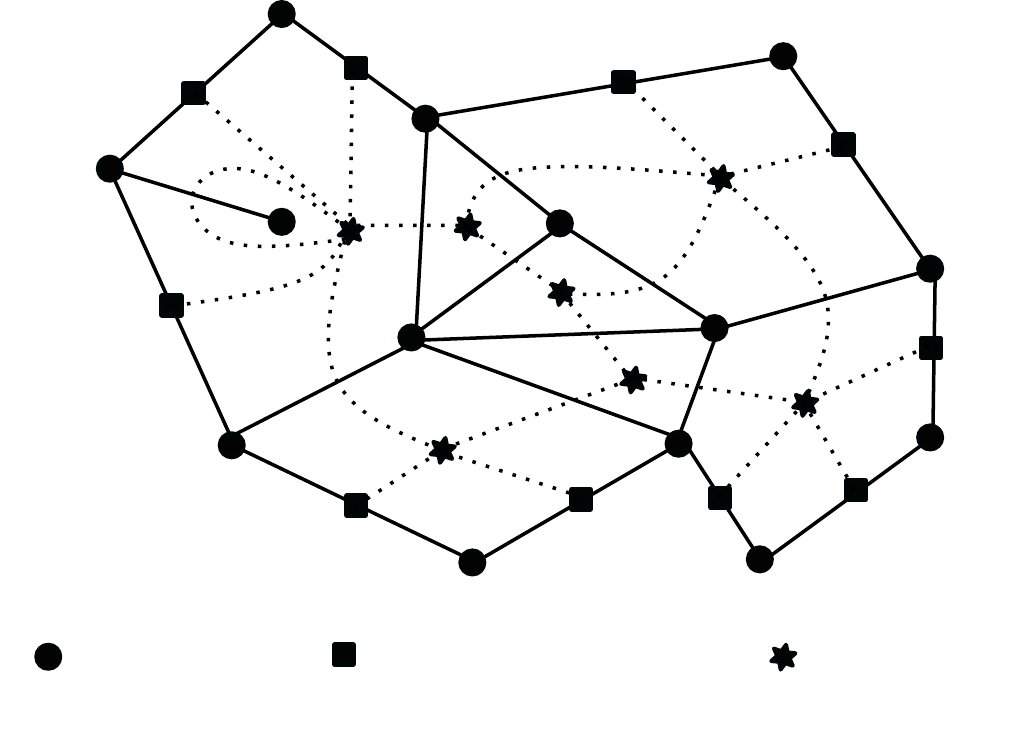}}%
    \put(0.09138426,0.07313038){\makebox(0,0)[lt]{\lineheight{1.25}\smash{\begin{tabular}[t]{l}$=V$\end{tabular}}}}%
    \put(0.3759283,0.07313038){\makebox(0,0)[lt]{\lineheight{1.25}\smash{\begin{tabular}[t]{l}$=V^{\perp,\partial}$\end{tabular}}}}%
    \put(0.80548315,0.07313038){\makebox(0,0)[lt]{\lineheight{1.25}\smash{\begin{tabular}[t]{l}$=V^{\perp, int}$\end{tabular}}}}%
    \put(0,0){\includegraphics[width=\unitlength,page=2]{admissible-dual.pdf}}%
    \put(0.08435327,0.00360291){\makebox(0,0)[lt]{\lineheight{1.25}\smash{\begin{tabular}[t]{l}$=E$\end{tabular}}}}%
    \put(0,0){\includegraphics[width=\unitlength,page=3]{admissible-dual.pdf}}%
    \put(0.37873939,0.00467716){\makebox(0,0)[lt]{\lineheight{1.25}\smash{\begin{tabular}[t]{l}$=E^\perp$\end{tabular}}}}%
  \end{picture}%
\endgroup%
}
\caption{An admissible planar graph and its dual graph.}
\label{admissible-dual}
\end{figure}  
\begin{remark}\label{rmk-dual-nc}
    Note that the dual graph of an admissible planar complex $G=(V,E,F)$ is not always connected. It is connected if and only if $G$ satisfies the property that two faces $f_1,f_2\in F$ share a vertex only if they share an edge.
\end{remark}
We next introduce notations relevant to the dual graph:
\begin{dfn} If $\alpha:E\to\mathbb R$ is a $1$-form on a graph $G=(V,E,F)$ and $G^\perp=(V^\perp, E^\perp)$ is the dual graph, we associate to it a \textbf{dual $1$-form} $\alpha^\perp:E^\perp\to\mathbb R$ defined by%
\begin{equation}\label{perpmap}
    \alpha^\perp(e^\perp):=\alpha(e)
\end{equation}
for all edges $e\in E$. Note that $(\pi\circ\alpha)^\perp=\pi\circ (\alpha^\perp)$, so we denote this form by $\pi\circ\alpha^\perp$.
\end{dfn}
\begin{dfn}
If $G=(V,E)$ is a bidirectional graph and $\gamma:E\to \mathbb R$ is a $1$-form, we define the {\bfseries divergence} $\mathsf{div}\, (\gamma):V\to\mathbb R$ as follows:
%
\[
\mathsf{div}\, (\gamma)(v):=\sum_{v'\in V: (v,v')\in E}\gamma(v,v').
\]
A $1$-form $\gamma:E\to\mathbb R$ has {\bfseries integral divergence} if $\mathsf{div}\, (\gamma)(v)\in\mathbb Z$ for all $v\in V^{int}$.
\end{dfn}
Observe that for the $\curl$ operation to be defined we need to include the faces $F$ in the planar complex $G$ (and these are also needed to determine the dual graph), while for the divergence operation (that we will use in the dual graph) vertices and edges are enough.

In order to translate hypotheses (h0), (h1) and (h2) in terms of the dual graph, we introduce the following.
\begin{dfn}
Given a $1$-form $\gamma$ and a choice of boundary vertices $V^\partial\subset V$, we define the {\bfseries total flux} through the boundary $V^\partial$ 
\[
\mathsf{Flux}(\gamma):=\sum_{v\in V^\partial}\mathsf{div}\,(\gamma)(v),
\]
and the {\bfseries boundary total variation}
\[
\|\gamma\|_{\partial}:=\sum_{v\in V^\partial}\sum_{x\sim v} |\gamma(v,x)|.
\]
\end{dfn}
With these definitions the discrete divergence theorem holds, i.e., there holds
\begin{equation}
    \label{eq-divergence-thm}
    \sum_{v\in (V)^{int}} \mathsf{div}\,(\gamma)(v)= -\mathsf{Flux}(\gamma).
\end{equation}
The following result is a direct consequence of the above definitions, and allows to compare hypotheses (h1),(h2) on a planar complex to their analogues on the dual graph. The easy proof is left to the reader.
\begin{lemma}\label{lem:dual} 
    Let $G=(V, E, F)$ be an admissible planar complex, and let $\alpha:E\to\mathbb R$ be a $1$-form on $G$. Let $G^\perp=(V^\perp, E^\perp)$ be the dual graph and $V^{\perp,\partial}$ be the corresponding boundary vertices. Then the following hold.
    \begin{enumerate}
        \item The mapping $\alpha\mapsto \alpha^\perp$ from $1$-forms over $(V,E)$ to their dual $1$-forms over $(V^\perp, E^\perp)$ is a bijection.
        \item For any face $f\in F$ there holds
        \begin{equation}\label{eq:divcurl}
        \mathsf{curl}\,( \alpha)(f)=\mathsf{div}\,(\alpha^\perp)(f^\perp).
        \end{equation}
        \item If $\alpha=du$ for some $u:V\to\mathbb R$ then (h1) holds for $u$ if and only if on $G^\perp$ with boundary vertices $V^{\perp,\partial}$ we have
        \begin{equation}\tag{$h1^\perp$}\label{eq:hyp1dual}
            \mathsf{Flux}(\pi\circ\alpha^\perp)=0\quad\textup{and}\quad \|\pi\circ\alpha^\perp\|_\partial \leq 1.
        \end{equation}
        \item If $\alpha=du$ for some $u:V\to\mathbb R$ then (h2) holds for $u$ if and only if on $G^\perp$ with boundary vertices $V^{\perp,\partial}$ we have
        \begin{equation}\tag{$h2^\perp$}\label{eq:hyp2dual}
            \mathsf{Flux}(\pi\circ \alpha^\perp) \in \{\pm 1\}\quad\textup{and}\quad \|\pi\circ\alpha^\perp\|_\partial = 1.
        \end{equation}
    \end{enumerate}
\end{lemma}
\section{On $1$-forms with integral divergence}\label{sec-1fid}
We next extend the results of \cite[Sec. 2.2]{petrache2015singular} on $1$-forms with integral divergence to a more general context. Note that \emph{the results of this section hold for general, possibly non-planar graphs, using only their combinatorial structure}.
\subsection{Notations and statements of the results}
We consider as before a connected bidirectional graph $G=(V,E)$, with selected boundary vertices $V^\partial\subset V$.
%
%
%
%
The following are the main results of this section. They are reminiscent of \cite[Prop. 2.5]{petrache2015singular}.
\begin{thm}\label{thm:removedipoles}
Consider a bidirectional graph $G=(V, E)$ with boundary vertices $V^\partial\subset V$.
Let $\alpha:E\to\mathbb R$ be a $1$-form with integral divergence and assume furthermore that
\[
\mathsf{Flux}(\alpha)=0,\quad \|\alpha\|_{\partial}\leq 1.
\]
Then, there exists a $1$-form $\gamma$ such that 
\begin{itemize}
\item $|\gamma|\leq |\alpha|$ for all edges in $E$ and $\gamma=\alpha$ over edges from $V^\partial\times V$,
\item $\mathsf{div}\,(\gamma)(v)=0$ for all $v\in V\setminus V^\partial$.
\end{itemize}

Furthermore, if $\|\alpha\|_\partial <1$, and $\mathsf{div}(\alpha)\neq 0$ on $V\setminus V^\partial$, then $\gamma$ can be chosen so that there exists an edge $e\in E$ at which $|\gamma(e)| < |\alpha(e)|$.
\end{thm}
\begin{thm}\label{thm:removedipoles1}
Consider a bidirectional graph $G=(V, E)$ with boundary vertices $V^\partial\subset V$.
Let $\alpha:E\to\mathbb R$ be a $1$-form with integral divergence and assume furthermore that
\[
\mathsf{Flux}(\alpha)
\in \{\pm1\},\quad \|\alpha\|_{\partial}=1.
\]
Then, there exist a $1$-form $\gamma$ and a vertex $x_0\in V\setminus V^\partial$ such that
\begin{itemize}
\item $|\gamma|\leq |\alpha|$ and $\gamma=\alpha$ over edges from $V^\partial\times V$,
\item $\mathsf{div}\, (\gamma)(x_0)=-\mathsf{Flux}(\alpha)$ and $\mathsf{div}\,(\gamma)(v)=0$ for all $v\in V\setminus (V^\partial\cup \{x_0\})$.
\end{itemize}
\end{thm}

\begin{remark}\label{r:specifica ipotesi}
Under the hypotheses of Theorem \ref{thm:removedipoles1} we have that
\[
1=\left\vert \mathsf{Flux}(\alpha)\right\vert =
\left\vert \sum_{v\in V^\partial}\sum_{x\sim v} \alpha(v,x)\right\vert \leq
\sum_{v\in V^\partial}\sum_{x\sim v} |\alpha(v,x)| = \|\alpha\|_{\partial}=1,
\]
and therefore the sign of $\alpha(v, x)$ with $v\in V^\partial$ and $x\sim v$ is constant, i.e.,
\[
\alpha(v, x) = \mathsf{Flux}(\alpha) |\alpha(v,x)| \quad \forall\; v\in V^\partial \; \forall\; x\sim v.
\]
In particular, for every $v_1, v_2\in V^\partial$ with $v_1\sim v_2$ we have that $\alpha(v_1,v_2)=0$.
\end{remark}



\begin{remark}\label{r:optimal}
The previous results are optimal for what concerns the hypotheses on the total flux and the total boundary variation of the forms, in the sense illustrated in the following examples. 

\medskip

\noindent \textit{Example of the optimatlity of Theorem \ref{thm:removedipoles}}.
Consider the graph with vertices $V:=\{a,b,A,B\}$, with $V^\partial:=\{a,b\}$ and $a\sim A, A\sim B, B\sim b$. Consider a $1$-form $\alpha:E\to \mathbb R$ such that 
    \[
        \alpha(a,A)=-\alpha(A,a)=\alpha(B,b)=-\alpha(b,B)=\frac12+\epsilon,\quad -\alpha(A,B)=\alpha(B,A)=\frac12 -\epsilon.
    \]
    Then $\mathsf{div}\,(\alpha)(A)=-1$ and $\mathsf{div}\,(\alpha)(B)=1$. We have thus
    \[
        \mathsf{Flux}(\alpha)=0,\quad \|\alpha\|_{\partial}=1+2\epsilon.
    \]
It is easy to check that if $|\gamma|\leq |\alpha|$ is a $1$-form with integral fluxes that coincides with $\alpha$ over $(a,A), (b,B)$ must actually be equal to $\alpha$. This shows that the hypothesis of Theorem \ref{thm:removedipoles} on $\|\alpha\|_{\partial}$ is sharp.

\medskip

\begin{figure}[htbp]
\centering
{\def\svgwidth{300pt}
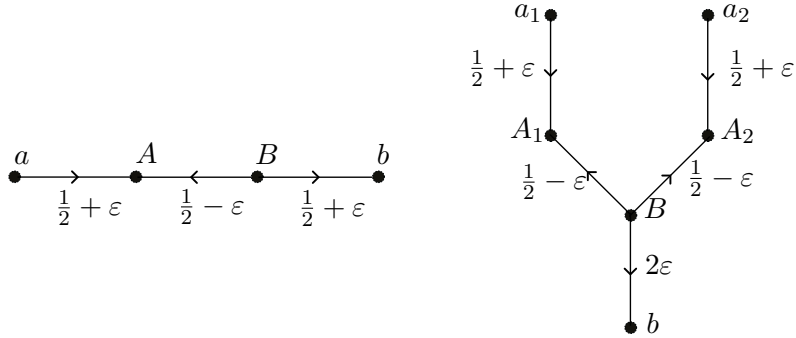}
\label{fig-examples1}
\caption{Examples illustration that the assumptions of Theorem \ref{thm:removedipoles} are Theorem \ref{thm:removedipoles1} are sharp.}
\end{figure}
\noindent \textit{Example of the optimality of Theorem \ref{thm:removedipoles1}}.
Consider the graph with vertex set $V:=\{a_1,a_2,A_1,A_2,B,b\}$ with $V^\partial:=\{a_1,a_2,b\}$ and $a_i\sim A_i, A_i\sim B, B\sim b$ (for $i=1,2$). We consider the $1$-form $\alpha:E\to\mathbb R$ determined by the following (again for $i=1,2$):
    \[
        \alpha(a_i,A_i)=\frac12+\epsilon,\quad \alpha(A_i,B)=-\frac12+\epsilon,\quad \alpha(B,b)=2\epsilon.
    \]
    Then $\mathsf{div}\,(\alpha)(A_i)=-1, \mathsf{div}\,(\alpha)(B)=+1$ and we have
    \[
        \mathsf{Flux}(\alpha)=+1,\quad \|\alpha\|_{\partial}=1+4\epsilon.
    \]
It is easy to check that if $|\gamma|\leq |\alpha|$ is a $1$-form with integral fluxes that coincides with $\alpha$ over boundary edges $(a_i,A_i), (B,b)$, then the values of $\gamma$ over the remaining edges are uniquely determined and we have $\gamma=\alpha$. This shows that the hypothesis of Theorem \ref{thm:removedipoles1} on $\|\alpha\|_{\partial}$ is sharp.

\noindent \textit{Example for Theorem \ref{thm:removedipoles1}} In this example we show the most basic situation in which it is not possible to remove a dipole without modifying our methods. Consider the graph with $V:=\{a_1,a_2, a_3, A_1, A_2, A_3, B\}$ and $a_i\sim A_i, A_i\sim B$ for $1\le i\le 3$, and with $V^\partial:=\{a_1,a_2,a_3\}$. We define the $1$-form $\alpha:E\to \mathbb R$ by 
\[
    \alpha(a_i, A_i)=\frac23, \quad \alpha(B,A_i)=\frac13,\quad\text{for }1\le i\le 3.
\]

\begin{figure}[h]
\centerline{{\def\svgwidth{200pt}
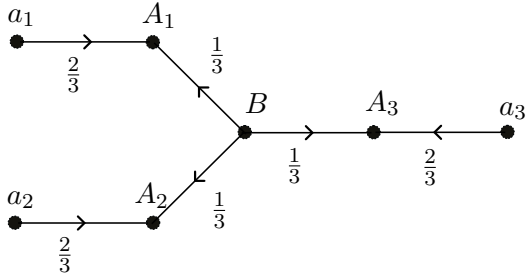}}
\caption{Example illustrating that in some cases it is not possible to remove the dipoles.}
\label{fig-examples2}
\end{figure}
Then $\mathsf{Flux}(\alpha)=\|\alpha\|_\partial=2$, and it is easy to check that if $|\gamma|\leq |\alpha|$ is a $1$-form with integral fluxes that coincides with $\alpha$ over $(a_i,A_i), 1\le i\le 3$, then we must have $\gamma=\alpha$. In particular, the statement of Theorem \ref{thm:removedipoles1} does not extend in a simple way to the case of $|\mathrm{Flux}(\alpha)|>1$. The observation of this example is the main motivation for open question \ref{q:bestupperbound}. 
\end{remark}

\subsection{Reduction to simpler graphs}\label{sec:auxhyp}
In view of the proof of Theorems \ref{thm:removedipoles} and \ref{thm:removedipoles1}, we show in this section that we can reduce to the case in which a few simplifying hypotheses hold for the graphs $(V,E)$. We start by reducing the proof to the case of connected graphs.
\begin{lemma}\label{lem:connectedhyp}
    The following reductions hold:
    \begin{enumerate}
        \item If Theorem \ref{thm:removedipoles} holds for connected bidirectional graphs $G=(V,E)$ with boundary vertices $V^\partial$, then its statement follows for all bidirectional graphs with boundary.
        \item If Theorems \ref{thm:removedipoles} and \ref{thm:removedipoles1} hold for connected bidirectional graphs with boundary, then Theorem \ref{thm:removedipoles1} holds for bidirectional graphs with boundary.
    \end{enumerate}
\end{lemma}
\begin{proof}
    Assume first that $G=(V,E)$ is a bidirectional graph with boundary vertices $V^\partial$, which can be decomposed into disjoint nonempty graphs $G_1=(V_1, E_1), G_2=(V_2, E_2)$ with boundaries $V^\partial_1, V^\partial_2$ respectively, such that $V^\partial = V^\partial_1\cup V^\partial_2$. 
    
    The hypotheses of Theorem \ref{thm:removedipoles} for $G$ imply for the restrictions of $\alpha$ to $E_1, E_2$, there holds $\mathsf{Flux}(\alpha|_{E_1})+\mathsf{Flux}(\alpha|_{E_2})=0$ and $\|\alpha|_{E_1}\|_\partial + \|\alpha|_{E_2}\|_\partial \leq 1$. Note that due to \eqref{eq-divergence-thm} and to the condition that $\alpha$ has integer divergence, $\mathsf{Flux}(\alpha|_{E_j})\in \mathbb Z$ and from the definitions it follows that $\mathsf{Flux}(\alpha|_{E_j})\leq \|\alpha|_{E_j}\|_\partial \leq 1$. These two facts imply that $\mathsf{Flux}(\alpha|_{E_j})=0$ and thus the hypotheses of Theorem \ref{thm:removedipoles} hold separately for $G_1, G_2$. If $G_1,G_2$ are not both connected, we can repeat the previous reasoning iteratively, until we reduce from $G$ to its connected components. Then if the thesis of Theorem \ref{thm:removedipoles} holds for each connected component of $G$, it directly follows for $G$ itself, concluding the proof of item 1.

    Now assume that $G, G_1, G_2$ are as above, and that the hypotheses of Theorem \ref{thm:removedipoles1} hold for a $G$. By a reduction analogous to the one for item 1, we now find that, up to interchanging the roles of $G_1, G_2$, we can assume that $\mathsf{Flux}(\alpha|_{E_1})=\|\alpha|_{E_1}\|_\partial=0$ and $\mathsf{Flux}(\alpha|_{E_2})=\mathsf{Flux}(\alpha)\in\{\pm 1\}, \|\alpha|_{E_2}\|_\partial =1$. If $G_1, G_2$ are not both connected, we can repeat the reduction iteratively until we reduce to the connected components of $G$. For all but one such components the hypotheses of Theorem \ref{thm:removedipoles} hold, and for the remaining one component, the hypotheses of Theorem \ref{thm:removedipoles1} hold. Then it is straightforward to verify that the conclusions for these two situations allow to build a global $1$-form $\gamma$ over $G$ satisfying the thesis of Theorem \ref{thm:removedipoles1}.
\end{proof}
As allowed by the reduction from Lemma \ref{lem:connectedhyp}, from now on we will assume that the considered graphs are connected. We next reduce to graphs with simpler local structure.
\begin{lemma}\label{lem:auxhyp}
Consider a connected bidirectional graph $G=(V, E)$, let $V^\partial\subset V$ denote the boundary points and let $\alpha:E\to\mathbb R$ be a $1$-form satisfying either the hypotheses of Theorem \ref{thm:removedipoles} or those of Theorem \ref{thm:removedipoles1}. We can then construct a new connected bidirectional graph $\overline{G}=(\overline V,\overline E)$, with boundary vertices $\overline V^\partial$ and a new $1$-form $\overline \alpha:\overline E\to\mathbb R$ such that the same hypotheses of the Theorems \ref{thm:removedipoles} and \ref{thm:removedipoles1} hold, and furthermore we have the following:
\begin{enumerate}[(R1)]
\item $\overline E\cap\big(\overline V^\partial\times \overline V^\partial\big)=\emptyset$;
\item for each vertex $v\in \overline V^\partial$, either
\[
\overline \alpha (v,x)>0\qquad \forall\; x\sim v,
\]
or 
\[
\overline \alpha (v,x)<0\qquad \forall\; x\sim v;
\]
\item for all $v\in \overline V \setminus \overline V^{\partial}$ we have that $\mathsf{div}\,( \overline\alpha) (v)\in \{0,\pm1\}$ and, if $\mathsf{div}\,(\overline\alpha)(v)\neq 0$, then either
\[
\overline \alpha (v,x)\geq 0\qquad \forall\; x\sim v,
\]
or 
\[
\overline \alpha (v,x)\leq 0\qquad \forall\; x\sim v.
\]
\end{enumerate}
Furthermore, if $\overline\gamma:\overline E\to\mathbb R$ is a $1$-form satisfying the conclusions of Theorems \ref{thm:removedipoles} and  \ref{thm:removedipoles1} for the given $\overline \alpha$, then it explicitly induces a $1$-form $\gamma:E\to\mathbb R$ which satisfies the conclusion of the  same theorems for the initial form $\alpha$.
\end{lemma}

\begin{proof}
We consider separately the reduction to the hypotheses \textit{(R1), (R2), (R3)}, describing in each case the way to recover the form $\gamma$ from $\overline{\gamma}$.

\medskip

\noindent \textit{(R1)}. It suffices to take $\overline V = V$, $\overline E := E \setminus (V^\partial\times V^\partial)$, and set $\overline{\alpha}:=\alpha\vert_{\overline E}$. Clearly, the graph $(\overline V, \overline E)$ satisfies \textit{(R1)}.
Moreover, $\mathsf{Flux}(\overline \alpha)=\mathsf{Flux}(\alpha)$ and under the hypotheses of Theorem \ref{thm:removedipoles} we have that $\|\overline \alpha\|_\partial \leq \|\alpha\|_\partial\leq 1$, while under the hypotheses of Theorem \ref{thm:removedipoles1} we have that $\|\overline \alpha\|_\partial =\|\alpha\|_\partial= 1$ (cp. Remark \ref{r:specifica ipotesi})

Finally, if $\overline\gamma$ satisfies the conclusions of the Theorems \ref{thm:removedipoles} and \ref{thm:removedipoles1} in $(\overline V, \overline E)$ for $\overline \alpha$, then the $1$-form
\[
\gamma(e):=
\begin{cases}
\overline \gamma(e) & \textup{if }\; e\in \overline{E},\\
\alpha (e) & \textup{if }\; e\in V^\partial\times V^\partial,
\end{cases}
\]
is such that $|\gamma|\leq|\alpha|$, $\gamma$ agrees with $\alpha$ on the boundary vertices $V^\partial$ and still satisfies the conclusion of the theorems.

\medskip 

\noindent \textit{(R2)}.  Without loss of generality we can assume \textit{(R1)} and we consider two copies of $V^\partial$, i.e., two sets $W^\partial_+, W^\partial_-$ with bijective maps $\pi_\pm : V^\partial \to W^\partial_\pm$. For simplifying the notation we denote $v_\pm := \pi_{\pm}(v)$ for every $v\in V^{\partial}$.
The set of vertices of the new graph $(\overline V, \overline E)$ is given by $\overline V:= \overline V^{int}\cup\overline V^\partial$ with
\[
\overline V^{int} := V^{int},\quad \overline V^\partial := \overline V^\partial_+\cup \overline V^\partial_-
\]
where $\overline V^\partial_\pm \subset W^\partial_\pm$ are given by the following condition:
\[
\overline V^\partial_\pm:=\Big\{v_\pm \in W^\partial_\pm : \exists\ x\in V,\ x\sim v\; \textup{such that }  \pm \alpha(v,x)>0\Big\}.
\]
The edges of $\overline E$ are then of three kinds: $(a,b)\in \overline E$ if and only if
\begin{itemize}
    \item $a, b\in \overline V^{int}=V^{int}$ and $(a, b)\in E$;
    \item $a=v_+\in \overline V^\partial_+$, $b\in \overline V^{int}$ and $(v,b)\in E$ with $\alpha(v,b)>0$, or symmetrically $b=v_+\in \overline V^\partial_+$, $a\in \overline V^{int}$ and $(v,a)\in E$ with $\alpha(v,a)>0$; 
    \item $a=v_-\in \overline V^\partial_-$, $b\in \overline V^{int}$ and $(v,b)\in E$ with $\alpha(v,b)<0$, or symmetrically $b=v_-\in \overline V^\partial_-$, $a\in \overline V^{int}$ and $(v,a)\in E$ with $\alpha(v,a)<0$.
\end{itemize}
In other words, $(\overline V,\overline E)$ is made by the interior points of $V$ with the corresponding edges, whereas the boundary points $\overline V^\partial$ correspond to the boundary points $v\in V^\partial$ labeled $\pm$ (eventually both $v_+$ and $v_-$) depending on the existence of edges $(v,x)\in E$ with $\mathsf{sgn}(\alpha(v,x))=\pm 1$. 

There is a natural projection map $\pi:\overline V\to V$ given by
\[
\pi(x):=
\begin{cases}
x & \textup{if } \ x\in \overline V^{int},\\
\pi_\pm^{-1}(x) & \textup{if } \ x\in \overline V^\partial_\pm.
\end{cases}
\]
and there is a 1:1 correspondence between edges, $\Pi:\overline E\to E$, given by
\[
\overline E \ni (a,b) \;\stackrel{\Pi}{\longmapsto}\; (\pi(a), \pi(b)) \in E.
\]
We define $\overline\alpha:\overline E\to\R$ as
\[
\overline\alpha(a,b):=\alpha\Big(\Pi(a,b)\Big).
\]
We then have that for every $v\in \overline V^{int}=V^{int}$
\begin{align*}
\mathsf{div}\,( \overline\alpha) (v) &= \sum_{(v,x)\in \overline E} \overline\alpha(v, x) 
= \sum_{(v,x)\in \overline E} \alpha\Big(\Pi(v,x)\Big)=\sum_{(v,x)\in \overline E} \alpha\Big(\pi(v),\pi(x)\Big)\\
&=\sum_{y\sim \pi(v),} \alpha(\pi(v), y)=
\mathsf{div}\, (\alpha) (\pi(v)) = \mathsf{div}\,(\alpha) (v).
\end{align*}
Therefore, the divergence of $\overline\alpha$ is unchanged in $\overline V^{int}=V^{int}$, and thus $\overline \alpha$ has integral divergence. 

By definition $\overline\alpha(v_+,x)>0$ for every $v_+\in \overline V^\partial_+$ and $x\sim v_+$ in $(\overline V, \overline E)$, and $\overline\alpha(v_-,x)<0$ for every $v_-\in \overline V^\partial_-$ and $x\sim v_-$. Moreover, the total flux and the total boundary variation of $\overline\alpha$ is unchanged with respect to those of $\alpha$.

Finally, if $\overline \gamma$ is a $1$-form satisfying the conclusion of the Theorems \ref{thm:removedipoles} and \ref{thm:removedipoles1} with respect to $\overline\alpha$, then the form
\[
\gamma(a,b) := \overline \gamma \Big(\Pi^{-1}(a, b)\Big) \qquad \forall\; (a,b)\in E,
\]
satisfies the same conclusions with respect to $\alpha$, because $|\overline \gamma| \leq |\overline \alpha|$ if and only if $|\gamma|\leq |\alpha|$, and $\mathsf{div}\, (\gamma) (v)=\mathsf{div}\,(\overline\gamma) (v)$ for all $v\in \overline V^{int}=V^{int}$.

\medskip 

\noindent \textit{(R3)}. The idea is similar to the previous case: we define the new graph $(\overline V, \overline E)$ by replacing each vertex $v\in V^{int}$ with $\mathsf{div}\,(\alpha)(v)=k$ and $|k|>1$, then we lift the graph by splitting the vertex $v$ into $k+1$ copies $v^{(0)},\dots,v^{(k)}$, and for every edge $(v,x)\in E$ we consider $k+1$ different edges $(v^{(j)},x)\in \overline E$.
We then define $\overline\alpha(v^{(j)},x)$ in such a way that
\[
\mathsf{div}\, (\overline\alpha) \big(v^{(j)}\big) =1 \quad j=1, \ldots, k,\qquad \mathsf{div}\, (\overline\alpha) \big(v^{(0)}\big) = 0,
\]
and
\[
\overline \alpha (v^{(j)},x) \, \alpha (v, x)\geq 0 \quad \textup{and}\quad
\sum_{j=0}^{k}\overline\alpha (v^{(j)},x) = \alpha (v, x) \quad\forall\;x\sim v.
\]
There exist several such liftings $\overline{\alpha}$ of $\alpha$. For example, if $\mathsf{div}\,(\alpha)(v)=k>0$, then we can split the vertices $x\sim v$ in two subsets:
\[
A_+:=\Big\{x\sim v : \alpha (v, x) \geq 0\Big\} \quad \textup{and}\quad A_-:=\Big\{x\sim v : \alpha (v, x) < 0\Big\},
\]
and we set $a_+:=\sum_{x\in A_+}\alpha (v,x)$, and
$a_-:=-\sum_{x\in A_-}\alpha (v,x)$ (hence, $a_+,a_-\geq 0$). Then, clearly $a_+-a_-=k$.
One we can then define $\overline\alpha$ in the following way:
\begin{gather*}
\overline\alpha (v^{(0)},x) :=
\begin{cases}
\alpha(v,x) & \textup{if } x \in A_-,\\
\frac{a_-}{k+a_-} \, \alpha(v,x)& \textup{if } x \in A_+,
\end{cases}\\
\overline\alpha (v^{(j)},x) :=
\begin{cases}
0 & \textup{if } x \in A_-,\\
\frac{\alpha(v,x)}{k+a_-} & \textup{if } x \in A_+,
\end{cases}
\end{gather*}
It is now straightforward to verify that with these definition $\overline \alpha$ satisfies the desired conclusions.

Finally, if $\overline\gamma$ is constructed in the lifted graph such to satisfy the conclusion of the theorems, then the projection on the original graph
\[
\gamma(v, x) := \sum_{j=0}^k\overline \gamma(v^{(j)},x)
\]
will satisfy the theses of the corresponding theorems too.\end{proof}
\subsection{Max-flow Min-cut}
In this section we introduce the notion of flow on graphs and recall the well-known \textit{Max-flow Min-cut Theorem} which will be used in the proof of the main results.

\begin{dfn}
Let $(V, E)$ be a connected bidirectional graph.
\begin{itemize}
\item A \textbf{simple path} is a sequence of edges
\begin{equation}\label{eq:path}
e_0 = (v_0,v_1), \; e_1 = (v_1,v_2), \ldots,\  e_m = (v_m,v_{m+1}) \in E,\;m\in \N,
\end{equation}
with $v_i\neq v_j$ for all $i\neq j$. We denote the path by $P=(e_0, \ldots, e_m)$, $e_j$ and $v_j$ are the edges and the vertices of $P$, respectively (we write $e_j\in P$ and $v_j\in P$, with a little abuse of notation), $v_0$ the starting vertex and $v_{m+1}$ the ending vertex.

\item A \textbf{flow} is a finite set of of couples
\[
\Phi:=\bigcup_{i=1}^N (P_i,m_i),
\]
with $P_i$ paths and $m_i\in (0, +\infty)$ the corresponding multiplicity. Given two disjoint subsets of vertices $V_1,V_2\subset V$, we say that a \textbf{flow $\Phi$ goes from $V_1$ to $V_2$}, and we write $\Phi:V_1\curvearrowright V_2$, if the starting vertices of $P_i$ belong to $V_1$ and the ending vertices belong $V_2$.
The \textbf{total flux} of a flow $\Phi=\bigcup_{i=1}^N (P_i,m_i)$ is the value
\[
T(\Phi):=\sum_{i=1}^N m_i.
\]
If $P$ starts in $V_1$ and ends in $V_2$ we also use the notation $P:V_1 \curvearrowright V_2$.

\item The \textbf{flow $1$-form} $\gamma_{\Phi}$ associated to a flow $\Phi=\bigcup_{i=1}^N (P_i,m_i)$ is defined by
\[
\gamma_{\Phi}(a, b) := \sum_{i:(a,b)\in P_i} m_i - \sum_{i:(b,a)\in P_i} m_i,
\]
with the convention that when the sum runs over the empty set the result is $0$.
\end{itemize}
\end{dfn}

\begin{remark}\label{rem-div-zero-flow}
The flow $1$-form $\gamma_{\Phi}$ associated to a flow $\Phi$ from $V_1$ to $V_2$ has zero divergence at all vertices in $V\setminus (V_1\cup V_2)$. Indeed, if ${\Phi}=\bigcup_{i=1}^N (P_i,m_i)$ then whenever $v\in P_i\setminus (V_1\cup V_2)$, then for each one of the indices $1\leq j_1<\cdots<j_{k_i}\leq m$ such that $v=v_{j_k}, 1\leq k \leq k_i$ in notation \eqref{eq:path}, we can associate to $v$ the two neighboring nodes; therefore we have
\begin{align*}
\mathsf{div}\,(\gamma_{\Phi}) (v) & = \sum_{v'\sim v}\gamma_{\Phi}(v,v') = \sum_{i=1}^N\sum_{(v,v')\in P_i}\gamma_{\Phi}(v,v')= \sum_{i=1}^N \sum_{k=1}^{k_i}\big(\gamma_{\Phi}(v,v_{j_k-1})+ \gamma_{\Phi}(v,v_{j_k+1})\big)\\
&=\sum_{i=1}^N\sum_{k=1}^{k_i}\big(-\gamma_{\Phi}(v_{j_k-1},v)+ \gamma_{\Phi}(v,v_{j_k+1}) = \sum_{i=1}^N k_i\big(-m_i+m_i\big)=0.
\end{align*}

\end{remark}

We introduce next the notion of capacity.

\begin{dfn}
Let $(V, E)$ be a connected bidirectional graph.
\begin{itemize}
\item A \textbf{capacity} $c:E\to [0,+\infty)$ is a functions such that 
\[
c(a,b) = c(b,a) \qquad \forall\;(a,b)\in E.
\]

\item We say that a flow $\Phi$ is \textbf{dominated} by a capacity function $c$, and we write $\gamma_{\Phi}\preccurlyeq c$, if
\[
\sum_{i : (a,b)\in P_i} m_i + \sum_{i : (b,a)\in P_i} m_i\leq c(a,b)\quad\forall\; (a,b)\in E.
\]
\item An edge $e=(a,b)\in E$ is \textbf{saturated} in a dominated flow if the equality holds
\[
\sum_{i:(a,b)\in P_i} m_i + \sum_{i:(b,a)\in P_i} m_i= c(a,b).
\]
\end{itemize}

\end{dfn}

Finally, we introduce the notion of cut.

\begin{dfn}
Let $(V, E)$ be a connected bidirectional graph and $V_1, V_2\subset V$ two disjoint subsets of vertices.
\begin{itemize}
\item A \textbf{cut} between $V_1$ and $V_2$ is a bidirectional subset of edges $C\subset E$ such that every path $P$ from $V_1$ to $V_2$ have to intersect $C$, i.e., $(a,b)\in C \Leftrightarrow (b,a) \in C$
and
\[
P\cap C\neq \emptyset \qquad \forall\; P:V_1 \curvearrowright V_2.
\]

    \item Given a capacity function $c:E\to [0,+\infty)$, the \textbf{capacity of a cut $C$} is the value
\[
c(C) := \frac12\sum_{e\in C} c(e).
\]
\end{itemize}

\end{dfn}

The following theorem is a version of a well-known result (see e.g. \cite{ford1956maximal}).

\begin{thm}[Max-flow Min-cut]\label{thm:mfmc}
Let $(V, E)$ be a connected bidirectional graph, let $V_1$, $V_2\subset V$ be two disjoint subsets of vertices and $c:E\to [0,+\infty)$ a capacity function.
Then, the maximal total flux over all flows from $V_1$ to $V_2$ dominated by the capacity $c$ coincides with the minimum capacity among the  cuts between $V_1$ and $V_2$:
\[
\max \Big\{ T(\Phi) : \Phi:V_1 \curvearrowright V_2,\;\gamma_{\Phi}\preccurlyeq c\Big\}
=
\min \Big\{ c(C) : C\subset V\;\textup{cut between $V_1$ and $V_2$}\Big\}.
\]
Moreover, if $C^\star$ is a minimum cut, then every edge $e\in C^\star$ is saturated by every maximal flow $\Phi^\star$ and, if $\Phi^{\star\star}$ is any other maximal flow with minimal cut $C^\star$, then for every $e\in C^\star$ there holds
\[
e\in \Phi^\star \Leftrightarrow -e\notin \Phi^\star\Leftrightarrow e \in \Phi^{\star\star}.
\]
\end{thm}

\begin{remark}
    \label{rem-optimal flow}
    As a consequence of the above result, for $V^\partial:= V_1\cup V_2$ Remark \ref{rem-div-zero-flow} shows that $\mathsf{Flux}(\gamma_{\Phi})=0$ for any flow form, and hence using the divergence theorem \eqref{eq-divergence-thm} for a maximal flow $\Phi^\star$ from $V^1$ to $V^2$ we get
    $$
   \sum_{v\in V^2} \mathsf{div}\,(\gamma_{\Phi^\star})(v)= -\sum_{v\in V^1} \mathsf{div}\,(\gamma_{\Phi^\star})(v)= c(C^\star)= T(\Phi^\star)=\sum_i m_i.
    $$
\end{remark}
\begin{remark}[orientation induced on a cut by a saturating flow]
    \label{rem-optimalflow2}
    With the same notation as in the previous remark, for $i=1,2$ we denote by $V^\star_i$ the set of vertices $x$ for which there exists a path $P: v\curvearrowright x$ with $v\in V_i, P\subset E\setminus C^\star$. In other words, in the graph $G=(V, E\setminus \{C^\star\})$ the set $V_i^\star$ is the union of connected components of vertices in $V_i$. We also set 
    \[
        C^\star_i:=\{\text{vertices of edges } e\in C^\star\}\cap V^\star_i.
    \]
    Then we observe that $e\in C^\star$ satisfies $e\in\Phi^\star$ if and only if it is directed from $C^\star_1$ to $C^\star_2$. The set of such directed edges will be denoted $\overrightarrow C^\star_{12}$. Note that, due to the last part of Theorem \ref{thm:mfmc}, this orientation is the same for any choice of the optimum flow $\Psi^\star$, in case this optimum flow is not unique.
\end{remark}

\subsection{Proofs of Theorem \ref{thm:removedipoles} and Theorem \ref{thm:removedipoles1}}\label{sec:thmremovedipoles}
%
%
\begin{proof}[Proof of Theorem \ref{thm:removedipoles}:]
By Lemma \ref{lem:connectedhyp} and Lemma \ref{lem:auxhyp} we can assume without loss of generality that $(V,E)$ is connected and that the assumptions $(R1)$ and $(R2)$ hold for $(V,E)$ and $\alpha$. 
We partition $V^\partial=V^{\partial}_{+}\cup V^{\partial}_{-}$ with 
\begin{equation}
    \label{eq-segno-bordo}
V^{\partial}_{\pm}:=\Big\{v\in V^\partial:\ \pm\alpha(v,x)>0\ \forall x\in V\Big\}.
\end{equation}
As $\mathsf{Flux}(\alpha)=0$, we find the following relations which we will use later:
    \begin{equation}\label{valalpha1}
    \xi_0:=\sum_{v\in V^\partial_+}\sum_{x\sim v}\alpha(v,x)
    =-\sum_{v\in V^\partial_-}\sum_{x\sim v}\alpha(v,x)=\frac12\|\alpha\|_\partial\leq \frac12.
\end{equation}
We apply Theorem \ref{thm:mfmc} with $V_1=V^{\partial}_{+},V_2=V^{\partial}_{-}$, and capacity function $c(a,b):=|\alpha(a,b)|$. Consider any maximal flow $\Phi^\star$, with $\gamma^\star=\gamma_{\Phi^\star}$ the associated flow $1$-form, and a minimum cut $C^\star$ and note that
\begin{equation}\label{e.flussi gamma* e alpha}
0\le \xi^\star:=c(C^\star)=\frac12\sum_{e\in C^\star}|\alpha(e)|=\sum_{v\in V^\partial_+}\sum_{x\sim v}\gamma^\star(v,x)\leq 
\sum_{v\in V^\partial_+}\sum_{x\sim v}\alpha(v,x)=\xi_0,
\end{equation}
where in the last inequality we use the constraint $|\gamma^\star|\le c=|\alpha|$, and the definition of $V^\partial_+$.

\medskip

With the notation as in Remark \ref{rem-optimalflow2}, consider the graph $G':=(V', E')$ in which $V'=V^\star_1\cup C^\star_2$ and $E'\subset E$ is induced by vertex set $V'$: 
\[
E':=\left\{e\in E:\ e \text{ has both endpoints in }V'\right\}.
\]
We define 
\[
V^{\prime\, \partial}:=V^{\partial}_+\cup C^\star_2.
\]
It is clear from the construction that each $v\in(V')^{int}= V'\setminus V^{\prime\, \partial} $ has the same set of neighbors in $G'$ as it did in $G$: indeed given $x$ such that $x\sim v$, or equivalently $(v,x)\in E$, if $v\notin C^\star$, we infer that $x$ is the end point of a path from $V^\partial_+$ not intersecting $C^\star$ (it is enough to reach $v$ with such a path and to prolong it with the edge $(v,x)$), in other words $(v,x)\in E'$. Otherwise if $v\in C^\star$ and $(v,x)\in E$, arguing by contradiction we conclude that $x\in V'$.
Thus, $\mathsf{div}\,(\alpha|_{G'})(v)=\mathsf{div}\,(\alpha)(v)\in\mathbb Z$ for all $v\in (V')^{int}$.
From Remark \ref{rem-optimalflow2}, edges $e=(v'',v')\in \overrightarrow C^\star_{12}$ point toward the exterior of $G'$ and satisfy $c(e)=\gamma^\star(e)$. Then, applying the divergence theorem to $\alpha|_{G'}$, we find 
\begin{eqnarray*}
\mathbb Z&\ni& -\sum_{v\in (V')^{int}}\mathsf{div}\,( \alpha)(v)=\sum_{v\in V^\partial_+}\sum_{x\sim v}\alpha(v,x)-
\sum_{e\in \overrightarrow C^\star_{12}}\alpha(e) \\
&=&\xi_0-\sum_{e\in \overrightarrow C^\star_{12}}\alpha(e).
\end{eqnarray*}
Since by Remarks \ref{rem-optimal flow} and \ref{rem-optimalflow2} we have
\[
\sum_{e\in \overrightarrow C^\star_{12}}|\alpha(e)|=\frac12\sum_{e\in C^\star}|\alpha(e)|=\xi^\star\leq \xi_0\le 1/2,
\]
there are only two possibilities: either
\[
\xi_0-\sum_{e\in \overrightarrow C^\star_{12}}\alpha(e) = 0,
\]
and therefore 
\[
\left\vert \sum_{e\in \overrightarrow C^\star_{12}}\alpha(e) \right\vert =\xi^\star= \xi_0\le 1/2;
\]
or
\[
\left\vert\xi_0-\sum_{e\in \overrightarrow C^\star_{12}}\alpha(e)\right\vert = 1,
\]
in which case
\[
\sum_{e\in \overrightarrow C^\star_{12}}\alpha(e)=-\frac12, \quad \textup{and}\quad  \frac12\sum_{e\in C^\star}|\alpha(e)|=\xi^\star= \xi_0 = 1/2.
\]
In either case we get $\xi^\star=\xi_0$. Recalling \eqref{e.flussi gamma* e alpha}, this in turn implies that $\gamma^\star=\alpha$ over $V\times V^{\partial}_{+}$, and thus also $\gamma^\star=\alpha$ over $V\times V^{\partial}_{-}$, in view of the chain of inequalities 
\begin{align*}
\sum_{v\in V^\partial_+}\sum_{x\sim v}\alpha(v,x) & = \sum_{v\in V^\partial_+}\sum_{x\sim v}\gamma^\star(v,x)=
-\sum_{v\in V^\partial_-}\sum_{x\sim v}\gamma^\star(v,x)\\
&\leq \sum_{v\in V^\partial_-}\sum_{x\sim v}|\alpha(v,x)|=-\sum_{v\in V^\partial_-}\sum_{x\sim v}\alpha(v,x)\\
&=\sum_{v\in V^\partial_+}\sum_{x\sim v}\alpha(v,x),
\end{align*}
where we used the flow properties of $\gamma^\star$, and in the last step we used the hypothesis $\mathsf{Flux}(\alpha)=0$. Thus the flow $\gamma^\star$ satisfies all the claimed properties, concluding the proof.

Finally, we prove the last claim of the theorem. Assume that $\|\alpha\|_\partial<1$, let $\Phi^\star=\sum_{i=1}^N (P_i,m_i)$ be an optimal flow as above, and assume that $|\alpha|=|\gamma_{\Phi^\star}|$ over all edges. Note that up to removing loops from the paths $P_i$, we may assume that paths $P_i, 1\le i\le N$ do not self-intersect. If $\alpha$ has a $\mathsf{div}(\alpha)(v)\neq 0$ at $v\in V\setminus V^\partial$, then since the divergence of $\alpha$ takes integer values we have $\sum_{x\sim v}|\alpha(v,x)|\geq 1$. On the other hand, paths $P_i$ all pass at most once through $v$. Let $I$ be the set of indices $1\le i\le N$ such that $P_i$ passes through $v$. Then for each $i\in I$ there exist exactly two oriented edges, denoted $(x_i,v),(v,y_i)\in E$, that belong to $P_i$, and that appear as consecutive edges in $P_i$. Then we have
\begin{eqnarray*}
    \sum_{x:x\sim v}|\gamma_{\Phi^\star}(v,x)|&=&\sum_{x:x\sim v} \left|\sum_{i\in I: x=x_i} m_i  - \sum_{j\in I: x=y_j}m_j\right|\leq \sum_{x:x\sim v} \left(\sum_{i\in I: x=x_i} m_i  +\sum_{j\in I: x=y_j}m_j\right)\\ 
    &=& 2\sum_{i\in I} m_i\leq 2 T(\Phi^\star)=2\xi^\star = 2\xi^0=\|\alpha\|_\partial<1 \leq \sum_{x:x\sim v}|\alpha(v,x)|.
\end{eqnarray*}
Thus there exists at least one $x\sim v$ such that $|\gamma_{\Phi^\star}(v,x)|<|\alpha(v,x)|$, as desired.

\end{proof}
\begin{proof}[Proof of Theorem \ref{thm:removedipoles1}:]
We assume that the graph $G$ is connected and that the auxiliary hypotheses \textit{(R1)-(R3)} from Lemma \ref{lem:auxhyp} for $\alpha$ are satisfied.
In particular, if we set
\[
V^\pm_\alpha := \Big\{ v\in V^{int} : \mathsf{div}\,(\alpha) (v) = \pm 1 \Big\},
\]
then by \textit{(R3)} we have
\begin{equation}\label{eq:valunderhyp}
\mathsf{Flux}(\alpha)=\# V^-_\alpha - \# V^+_\alpha.
\end{equation}
Without loss of generality we can assume that $\mathsf{Flux}\ (\alpha)=-1$, as for the other case it suffices to apply the same proof to $-\alpha$.
As stressed in Remark \ref{r:specifica ipotesi}, under the assumptions of the theorem we have that $V^\partial_+=\emptyset$ (where $V^\partial_\pm$ are defined in \eqref{eq-segno-bordo}). Moreover by \eqref{eq:valunderhyp} we have that $V_\alpha^+\neq\emptyset$ and we prove the result proceeding by induction over the cardinality of this set. 

\medskip

If $\# V_\alpha^+=1$ by the fact that $\mathsf{Flux} (\alpha)=-1$ it follows using \eqref{eq:valunderhyp} that $\#V_\alpha^-=0$, and we may take $\gamma=\alpha$, completing the proof. 

\medskip 

For the inductive step, we start by fixing an arbitrary point $x_0\in V^+_\alpha$ and we apply Theorem \ref{thm:mfmc} with capacity function $c=|\alpha|$ and $V_1:=V^{\partial}_{-}=V^\partial, V_2=\{x_0\}$.
Let $\Phi^\star,C^\star$ be a maximal flow and a minimal cut, and denote by $\xi^\star$ the optimum value $\xi^\star=c(C^\star)=T(\Phi^\star)$.
Since the flow $1$-form satisfies $\gamma_{\Phi^\star}$ we have $|\gamma_{\Phi^\star}|\leq |\alpha|$, we obtain from Remark \ref{rem-optimal flow} that
\begin{align*}
0\le \xi^\star &= -\sum_{v\in V^\partial_-}\sum_{x\sim v}\gamma_{\Phi^\star}(v,x)\leq \sum_{v\in V^\partial_-}\sum_{x\sim v}|\alpha(v,x)|= - \sum_{v\in V^\partial_-}\sum_{x\sim v}\alpha(v,x)=\mathsf{Flux}(\alpha)=1.
\end{align*}
Then, if we have that $\xi^\star=1$, then it follows that
\[
-\sum_{v\in V^\partial_-}\sum_{x\sim v}\gamma_{\Phi^\star}(v,x)=- \sum_{v\in V^\partial_-}\sum_{x\sim v}\alpha(v,x)=\|\alpha\|_\partial=1, 
\]
and therefore $\gamma_{\Phi^\star}$ is a $1$-form satisfying the conclusion of the theorem, because
\begin{itemize}
\item $|\gamma_{\Phi^\star}|\leq |\alpha|$,
\item $\mathsf{div}\,( \gamma_{\Phi^\star})(v)=0$ for all $v\in V^{int}\setminus \{x_0\}$ thanks to Remark \ref{rem-div-zero-flow} and $\mathsf{div}\,( \gamma_{\Phi^\star})(x_0)=1$ as a consequence of Remark \ref{rem-optimal flow},
\item $\gamma_{\Phi^\star}(v,x)= \alpha(v,x)$ for all $v\in V^\partial_-$ and all $x\sim v$.
\end{itemize}
   
\medskip

We consider the remaining case, namely $0\le \xi^\star<1$. We argue as in the proof of Theorem \ref{thm:removedipoles}. We set $G':=(V', E')$ to be the subgraph of $G$ which roughly speaking is induced by the vertices in $C^\star$ and the vertices that are not separated from $x_0$ upon removing the cut $C^\star$. More precisely, we take $V':=V^\star_2\cup C^\star_1$ with notations as in Remark \ref{rem-optimalflow2} for the choices $V_1=V_-^\partial=V^\partial, V_2=\{x_0\}$, and where $E'$ are the edges of $E$ with both endpoints in $V'$. Then for $G'$ consider the boundary vertices
\[
V^{\prime\, \partial}:=C^\star_1.
\]

It is easy to see that $\alpha|_{G'}$ has integer divergence over $V'\setminus V^{\prime\, \partial}$ and moreover
\[
    \|\alpha|_{G'}\|_\partial = c(C^\star)= \xi^\star<1.
\]
In particular, it follows that $|\mathsf{Flux}(\alpha\vert_{G'})| <1$, and since by the divergence theorem we have that 
\[
\mathsf{Flux}(\alpha\vert_{G'}) = -\sum_{v\in V'} \mathsf{div}\,( \alpha)(v) \in \Z,
\]
we deduce that $\mathsf{Flux}(\alpha\vert_{G'})=0$.
We thus may apply Theorem \ref{thm:removedipoles} to $\alpha|_{G'}$ and find a $1$-form $\gamma':E'\to\mathbb R$ such that $|\gamma'|\leq|\alpha|$ and $\gamma'=\alpha$ over $C^\star$.
   
Now we define a $1$-form $\alpha':E\to\mathbb R$ as follows:
\[
\alpha'(e)=\left\{\begin{array}{ll} 
\gamma'(e)&\text{ if }e\in E',\\
\alpha(e)&\text{ if }e\in E\setminus E'.
\end{array}\right.
\]
We have that 
\begin{itemize}
\item $\alpha'=\alpha$ over all $(v,x)$ such that $v\in V^\partial$ and $x\sim v$, thus $-\mathsf{Flux}(\alpha')=\|\alpha'\|_\partial=1$;
\item $\mathsf{div}\,( \alpha')(v')=\mathsf{div}\, ( \gamma')(v')=0$ for $v'\in V^\star_2\setminus\{x_0\}$ and $\mathsf{div} \,(\alpha')(v)=\mathsf{div} \,(\alpha)(v)\in\mathbb Z$ over $V^{\prime\, \partial}$ (here we use that $\gamma'=\alpha$ over $V^{\prime\, \partial}$);
\item furthermore,
\begin{gather*}
V^+_{\alpha'}\cap (V^{int}\setminus (V')^{int})= V^+_\alpha\cap (V^{int}\setminus (V')^{int}),\\
\# V^+_{\alpha'}\cap (V')^{int}=0<\# V^+_\alpha\cap (V')^{int},    
\end{gather*}
as now $\alpha'$ has no positive charge at $x_0$ anymore.
\end{itemize}
By our inductive hypothesis, we can thus find $|\gamma''|\leq \alpha'$ and $y\in V^+_{\alpha'}$ such that $\gamma''$ has zero divergence away from $y$ and $\gamma''=\alpha'$ over $V^\partial$. As $|\alpha'|\leq |\alpha|$ and $\alpha'=\alpha$ over $V^\partial$, we find that $\gamma''$ satisfies the desired properties of the theorem, concluding the proof.
\end{proof}
%
%
%
%
%
\section{Proof of Theorem \ref{thm:planargraphSD}}\label{sec:dualityend}
In this section we give the last ingredient needed to deduce Theorem \ref{thm:planargraphSD}: the next proposition guarantees that in the process of removing dipoles provided in Theorem \ref{thm:removedipoles} and Theorem \ref{thm:removedipoles1} we can keep the boundary condition unchanged.
In all this section  $(V,E,F)$ and $u$ (and the related notations) are the ones of Theorem \ref{thm:planargraphSD}. 
\begin{proposition}\label{prop:bdryvalue}
    Let $(V,E,F)$ be an admissible planar complex with boundary complex $(V^\partial, E^\partial)$, $u:V\to \mathbb R$ be a real-valued function, and $\widetilde\alpha:E\to[-1/2,1/2]$ be a $1$-form with values in $[-1/2,1/2]$, satisfying the following:
    \begin{enumerate}
        \item There exists $e_0\in E^\partial$ such that for all $e\in E^\partial\setminus\{e_0\}$ there holds $\pi(du(e)) = du(e)$.
        \item There exists $f_0\in F$ such that for all $f\in F\setminus\{f_0\}$ there holds $\mathsf{curl}\,(\widetilde\alpha)(f)=0$.
        \item For all $e\in E^\partial$ there holds $\pi(du(e))=\widetilde \alpha(e)$.
    \end{enumerate}
    Then there exists a function $\widetilde u:V\to\mathbb R$ such that
    \begin{equation}\label{eq:utilde}
        \widetilde u(v)= u(v)\quad \forall v\in V^\partial\quad\textup{and}\quad  \pi(d\widetilde u(e))= \widetilde \alpha(e)\quad \forall e\in E.
    \end{equation}
\end{proposition}
%
%

%
\begin{proof} 
We first consider the case that $\mathsf{curl}\,(\widetilde \alpha)=0$, while the general case reduces to this curl-free case by a suitable cut in the planar complex $(V, E,F)$. The useful fact to keep in mind is the "discrete Green's theorem" below (recall notation \eqref{eq:alphaofedges}):
\begin{equation}\label{eq:curlzerocircuitzero}
    \text{If $\mathsf{curl}\,(\widetilde \alpha)=0$ over $F$, then $\tilde\alpha(C)=0$ for any closed path $C$ in the graph.}
\end{equation}
To prove this, we may reduce to the case of $C$ being a simple closed path, i.e., a path without repeated vertices, as in general one can write $C$ as a union of simple closed paths. Furthermore, we can restrict to $C$ having at least $3$ vertices, as paths of the form $(v,v'),(v',v)$ automatically satisfy the property because $\widetilde \alpha$ is alternating. In this case, we assume $C$ to be oriented counterclockwise and, as $(V,E,F)$ is admissible, it follows that the interior of the simple closed curve induced by $C$ is a union of faces $f_1,\dots,f_k\in F$. Note that $\sum_{i=1}^k\mathsf{curl}\,(\widetilde\alpha)(f_i)=0$, and on the other hand, note that for any oriented edge $e\in E$ contained in the interior of $C$, there exist two faces $f_i,f_j$ such that $e\in\partial f_i, -e\in\partial f_j$ and $e,-e\notin \partial f_l$ for all $1\le \ell\le k, \ell\neq i,j$. Every edge in $C$ belongs to exactly one face $f_i$ each, and the (counterclockwise) orientations of $C$ and of $\partial f_i$ coincide on such edges. Thus, in summary,
\[
    0=\sum_{i=1}^k \mathsf{curl}\,(\widetilde \alpha)(f_i)=\sum_{i=1}^k\sum_{e\in\partial f_i}\widetilde\alpha(e)=\tilde\alpha(C), 
\]
as desired.

\textbf{Step 1: The case of $\mathsf{curl}\,(\widetilde \alpha)=0$ on all faces, including $f_0$.} 
We will use \eqref{eq:curlzerocircuitzero} in order to define $\widetilde u$ by extending it along a path. We fix a closed path $P_0=(v_0,v_1,\dots, v_n,v_0)$ which passes at least once through each $v\in V$, and such that $v_0\in V^\partial$. Then define $\widetilde u(v_0):=u(v_0)$ and iteratively for each $0\le i< n$ set 
\begin{equation}\label{eq:defineutilde}
    \widetilde u(v_{i+1}):=\widetilde u(v_i) + \widetilde \alpha(v_i,v_{i+1}) = \widetilde u(v_{j_0}) +\sum_{k=j_0}^i \widetilde \alpha(v_k, v_{k+1}), \quad \forall\ 0\le j_0 \le i,
\end{equation}
in which the second equality can be proved inductively and is included for future reference. Note that $\widetilde u$ is then well-defined on all vertices $v\in V$: indeed, if a vertex $v_i$ is visited more than once, i.e., there is $j>i$ such that $v_j=v_i$, then $(v_i,v_{i+1}, \dots v_{j-1},v_j)$ is a closed path, and we get
\[
    \widetilde u(v_j)\stackrel{\text{\eqref{eq:defineutilde}
    }}{=}\widetilde u(v_i)+\sum_{k=i}^{j-1} \widetilde \alpha(v_k,v_{k+1})\stackrel{\text{\eqref{eq:curlzerocircuitzero}}}{=}\widetilde u(v_i).
\]
Now note that thanks to \eqref{eq:curlzerocircuitzero} and \eqref{eq:defineutilde}, $d\widetilde u(e)=\widetilde \alpha(e)$ on all $e\in E$. We now verify the theses of the proposition. First $(V^\partial, E^\partial\setminus\{e_0\})$ is connected, $du= d\widetilde u$ on the edges $E^\partial\setminus\{e_0\}$, and we have $\widetilde u(v_0)=u(v_0)$, thus $u=\widetilde u$ over $V^\partial$. The fact that $\widetilde\alpha$ takes values in $[-1/2,1/2]$ allows to obtain $\widetilde\alpha=\pi\circ \widetilde\alpha=\pi\circ d\widetilde u$ on all edges, completing the proof. For future reference, note that, in this case, we also have $\widetilde \alpha(e_0)=du(e_0)= \pi(du(e_0))$.

\textbf{Step 2: The case of $\mathsf{curl}\,(\widetilde\alpha)(f_0)\neq 0$.} In this case denote $e_0=(v_0, v_1)$, which we assume oriented in counterclockwise direction along the boundary of the complex. We consider the shortest path $P$ in the graph $(V, E)$ from $v_0$ to $\partial f_0$, and build a new domain with boundary given by the realization of the counterclockwise path in the boundary complex going from $v_1$ to $v_0$, followed by $P$, $\partial f_0$, and the inverse path of $P$. 

To this domain we associate the complex $(V', E', F')$, in which we duplicate edges and vertices from $P$, including the endpoint vertex $w_0$ of $P$ belonging to $\partial f_0$. We denote by $\overline P$ the inverse path of $P$ formed by duplicated vertices and edges. In particular, the new boundary complex is 
\[
    (V^{\prime\, \partial},E^{\prime\, \partial})=(V^\partial\cup(V^P\cup V^{\overline P}\cup V^{\partial f_0}), E^{\partial} \cup (E^P\cup E^{\overline P}\cup E^{\partial f_0})),
\]
where $V^P$, $E^P$ are the vertices and edges of $P$ (and similarly for $\overline P$), $V^{f_0}$ is obtained from the set of vertices of $f_0$ after adding a copy $\overline w_0$ of $w_0$, and $E^{\partial f_0}$ is the set of edges of $\partial f_0$, denoted $(\overline w_0, w_n),(w_n,w_{n-1}),\dots,(w_1,w_0)$, oriented so that, before duplication, edges $(w_i,w_{i+1})$ were in counterclockwise order along $\partial f_0$. We then extend $\widetilde\alpha=\pi\circ\widetilde\alpha$ on duplicated edges by taking the same value as on the original ones, in particular the property $\mathsf{curl}\,(\widetilde\alpha)=0$ holds on the new complex. We then extend $u$ to $\hat u$ such that $\hat u( v_0)=u(v_0)$ and $\hat u=u$ on $V^\partial\cap V^{\prime\, \partial}$, and extend $\hat u$ along the path $P\cup \partial f_0\cup \overline P$, so that $d\hat u=\widetilde \alpha$ along the corresponding edges. Then $\pi\circ d\hat u=d\hat u$ along these edges, due to the property that $\widetilde \alpha=\pi\circ\widetilde \alpha$ as $\widetilde \alpha$ takes values in $[-1/2, 1/2]$. Note that after extension, if $\overline v$ is the duplicate of a vertex $v\in P$, then 
\begin{equation}\label{eq:uduplicated}
    \hat u(\overline v)-\hat u(v)=\mathsf{curl}\,(\widetilde\alpha)(f_0)\in\mathbb Z,\quad\text{ therefore }\quad\pi\big(\hat u(\overline v)-\hat u(v)\big)=0.
\end{equation}
In particular, by applying this for $v=v_0$ we find $\pi(d\hat u(\overline v_0, v_1))=\pi(du(v_0,v_1))$, and if in the original graph $v_{-1}$ was the neighbor of $v_0$ in clockwise direction along the boundary, then $d\hat u(v_{-1}, \overline v_0)=du(v_{-1},v_0)$.

We now apply the case of Step 1 with function $\hat u$ and $1$-form $\widetilde \alpha$ on the new complex, and obtain $\widetilde u: V'\to\mathbb R$ such that $\widetilde u=\hat u$ and $d\widetilde u=\widetilde \alpha=d\hat u$ on the boundary. We can now define (by a slight abuse of notation) $\widetilde u:V\to\mathbb R$ on the original complex by disregarding the values on duplicated vertices. Then the first equality in \eqref{eq:utilde} follows by definition since $\widetilde u=\hat u$ over non-duplicated vertices $V^\partial\setminus\{v_0\}$. Finally, we check the second equality from \eqref{eq:utilde}. Fix $v\in P$ and $w\in V\setminus P$. If $(v,w)\in E'$ then \eqref{eq:utilde} follows directly. If $(\overline v, w)\in E'$, then $\pi(d\widetilde u(v,w))=\pi(d\hat u(\overline v,w))=\widetilde\alpha(v,w)$. 
\end{proof}

\subsection{Conclusion of the proof of Theorem \ref{thm:planargraphSD}}\label{conclusions}
With the previous ingredients, we are now in the position to prove our main theorem.

\begin{proof}[Proof of Theorem \ref{thm:planargraphSD}:]
Starting with $u:V\to\mathbb R$ as in the theorem, we first apply Lemma \ref{lem:dual} to the $1$-form $du$, so that hypotheses (h1), (h2) translate respectively to \eqref{eq:hyp1dual} and \eqref{eq:hyp2dual}
for $(du)^\perp$. In these two cases, we apply respectively Theorem \ref{thm:removedipoles} and  Theorem \ref{thm:removedipoles1} to the form $\alpha=\pi\circ(du)^\perp$, obtaining a new form $\gamma$. Note that by the definition of $\pi$, the $1$-form $\pi\circ(du)^\perp$ takes values in $[-1/2, 1/2]$ and thus so does $\gamma$. We then apply Proposition \ref{prop:bdryvalue} with $\widetilde\alpha:=\gamma$, where the hypotheses of that proposition hold due to hypothesis (h0), and to the properties of $\gamma$ given in the thesis of either Theorem \ref{thm:removedipoles} or Theorem \ref{thm:removedipoles1}. The function $\widetilde u$ obtained from Proposition \ref{prop:bdryvalue} then verifies the properties required in Theorem \ref{thm:planargraphSD}, as desired. 
\end{proof}
%
%

\textbf{Acknowledgements:} A.G. acknowledges the financial support of PRIN 2022J4FYNJ “Variational methods for stationary and evolution problems with singularities and interfaces”, PNRR Italia Domani, funded by the European Union via the program NextGenerationEU, CUP B53D23009320006. Views and opinions expressed are however those of the authors only and do not necessarily reflect those of the European Union or The European Research Executive Agency. Neither the European Union nor the granting authority can be held responsible for them.

M.P. was supported by the Centro Nacional de Inteligencia Artificial and by Chilean Fondecyt grant 1210426 entitled "Rigidity, stability and uniformity for large point configurations". 

E.S. was supported by ERC-STG grant HiCoS "Higher
Co-dimension Singularities: Minimal Surfaces and the Thin Obstacle Problem”, Grant n. 759229.

\appendix
\section{Results for higher boundary total variation}\label{sec:openproblems}

As shown in Remark \ref{r:optimal}, the hypothesis on the total variation of the boundary datum in Theorems \ref{thm:removedipoles} and \ref{thm:removedipoles1} cannot be relaxed without changing the conclusions.
On the other hand one can still obtain some weaker results using similar arguments. The following theorem is given as an example.

\begin{thm}\label{thm:removedipoles2}
Consider a connected bidirectional graph $G=(V, E)$ with boundary vertices $V^\partial\subset V$.
Let $\alpha:E\to\mathbb R$ be a $1$-form with integral divergence and assume furthermore that
\[
\mathsf{Flux}(\alpha)\in \{\pm1\},\quad \|\alpha\|_{\partial}\in (1, 2].
\]
Then there exist a $1$-form $\gamma$ and a vertex $x_0\in V\setminus V^\partial$ such that
\begin{itemize}
\item $|\gamma|\leq 3|\alpha|$ and $\gamma=\alpha$ over edges from $V^\partial\times V$,
\item $\mathsf{div}\, (\gamma)(x_0)=-\mathsf{Flux}(\alpha)$ and $\mathsf{div}\,(\gamma)(v)=0$ for all $v\in V\setminus (V^\partial\cup \{x_0\})$.
\end{itemize}
\end{thm}
\begin{remark}\label{r:optimal2}
Consider the graph from the second example in Remark \ref{r:optimal}. Besides $\alpha$ itself, it is easy to check that the only forms satisfying the second conclusion of Theorem \ref{thm:removedipoles2} are $\gamma_1,\gamma_2$ given by
\[
\gamma_i(A_i,B)=\frac12+\epsilon,\quad \gamma_i=\alpha\text{ on all remaining edges}, \quad i=1,2.
\]
Note that have 
\[
\mathsf{Flux}(\alpha)=1,\quad \|\alpha\|_{\partial} = 1+ 4\eps, \quad
|\gamma_i(A_i,B)|=\frac{1+2\epsilon}{1-2\epsilon} |\alpha(A_i,B)|.
\]
For $\epsilon\leq 1/4$ we get $\|\alpha\|_{\partial}\leq 2$ and $\frac{1+2\epsilon}{1-2\epsilon}\leq 3$, which is coherent with Theorem \ref{thm:removedipoles2}, and the equality for $\epsilon=1/4$, thus showing the necessity of the hypothesis $\|\alpha\|_{\partial}\leq 2$.
\end{remark}
\medskip

\begin{proof}[Proof of Theorem \ref{thm:removedipoles2}:]
We assume that the auxiliary conditions \textit{(R1)-(R3)} from Section \ref{sec:auxhyp} hold for $\alpha$ and, without loss of generality, we assume $\mathsf{Flux}(\alpha)=1$.
We partition $V^\partial=V^{\partial}_{+}\cup V^{\partial}_{-}$ with 
\[
V^{\partial}_{\pm}:=\Big\{v\in V^\partial:\ \pm\alpha(v,x)>0\ \forall x\in V,\ x\sim v\Big\},
\]
and set 
\[
\xi_+ := \sum_{v\in V^\partial_+,\;x\sim v}\alpha(v,x), \qquad  \xi_- := -\sum_{v\in V^\partial_-,\;x\sim v}\alpha(v,x).
\]
Then, 
\[
\xi_+-\xi_- = \mathsf{Flux}(\alpha) = 1, \qquad \xi_++\xi_- = \|\alpha\|_{\partial} \leq 2.
\]
In particular, it follows that $\xi_-\leq\frac12$. We apply Theorem \ref{thm:mfmc} with $c=|\alpha|, V_1=V^{\partial}_+, V_2=V^{\partial}_-$, obtaining optimal flow and cut $\Phi^\star, C^\star$.
Set $\xi^\star:=c(C^\star)$. Clearly, since $C^\star$ separates $V^{\partial}_+$ from $V^{\partial}_-$ and $|\gamma_{\Phi^\star}|\le |\alpha|$, we have $\xi^\star\leq \xi_+\le 1/2$. Arguing exactly as in the proof of Theorem \ref{thm:removedipoles}, we infer  that $\xi^\star=\xi_-$. 
In particular, we get that $\gamma_{\Phi^\star}(e)=\alpha(e)$ over $e\in E\cap(V^{\partial}_-\times V)$, but we only know that $0\le \gamma_{\Phi^\star}(e)\le\alpha(e)$ over $e\in E\cap (V^\partial_+\times V)$, and in general the second inequality will be strict.

To conclude, set $\widetilde \alpha:=\alpha-\gamma_{\Phi^\star}$. The following properties can be checked directly:
\begin{itemize}
\item $\mathsf{div}\,(\widetilde \alpha)(v)=\mathsf{div}\,(\alpha)(v)\in\mathbb Z$ for $v\in V^{int}$,
\item $\mathsf{Flux}(\widetilde \alpha)=\mathsf{Flux}(\alpha)=1$ and $\|\widetilde\alpha\|_\partial=1$.
\end{itemize}
We then apply Theorem \ref{thm:removedipoles1} to $\widetilde\alpha$, and for the obtained $1$-form $\tilde \gamma$ we have $\widetilde\gamma(e)=\widetilde\alpha(e)$ for all $e\in E\cap( V^\partial\times V)$ and
\[
|\widetilde\gamma|\leq |\widetilde\alpha|\leq |\alpha| + |\gamma_{\Phi^\star}|\leq 2|\alpha|,
\]
and there exists $v_0\in V^{int}$ with  
\[
\mathsf{div}\,(\widetilde\gamma)(v_0) = -1,\quad \mathsf{div}\,(\widetilde\gamma)(v) = 0\; \forall v\in V^{int}\setminus \{x_0\}.
\]
Then the $1$-form $\gamma:=\gamma_{\Phi^\star}+\widetilde\gamma$ has divergence equal to the one of $\widetilde\gamma$, satisfies $\gamma=\gamma_{\Phi^\star}+\widetilde\alpha=\alpha$ over the edges with a vertex on the boundary $V^\partial$, and by triangular inequality
\[
|\gamma|\le |\gamma_{\Phi^\star}| + |\widetilde\alpha|\le 3|\alpha|,
\]
thus concluding the proof.
\end{proof}

For $M,K>0$ consider the following statement about a $1$-form $\alpha:E\to\mathbb R$ with integral divergence, defined over a connected bidirectional graph $G=(V,E)$:
\[
(P_{M,K}):\quad 
\left\{\begin{array}{ccc}\text{If $\mathsf{Flux}(\alpha)\ge 0$ and $\|\alpha\|_\partial\leq M$, then there exists a $1$-form $\gamma$ such that}\\[3mm]
\left.\begin{array}{lll}
\text{(1) $\gamma=\alpha$ on edges from $V^\partial\times V$,}\\
        \text{(2) $\gamma$ has integral divergence,}\\
        \text{(3) $V_\gamma^+=\emptyset,\quad V_\gamma^-\subseteq V_\alpha^-$,}\\
\text{(4) $|\gamma|\leq K|\alpha|$.}
        \end{array}\right.
        \end{array}\right.
\]

The above discussion naturally raises the following question: 
\begin{question}\label{q:bestupperbound}   
    Find the \textbf{best upper bound} function $K^\star:[0,+\infty)\to[0,+\infty]$ such that property $(P_{M,K^\star(M)})$ is true for all connected bidirectional graphs $G=(V,E)$ and for all $1$-forms $\alpha:E\to\mathbb R$ with integral divergence. \textbf{Is $K^\star(M)$ of polynomial growth as $M\to \infty$?}
\end{question}
It is trivial to see that we have \textbf{$K^\star(0)=0$} by simply taking $\gamma=0$ and that \textbf{$K^\star(\cdot)$ is increasing}. Our theorems \ref{thm:removedipoles} and \ref{thm:removedipoles1} show that $K^\star|_{(0,1]}\leq 1$, and the result of Theorem \ref{thm:removedipoles2} shows that $K^\star|_{(1,2]}\leq 3$. Furthermore, Remark \ref{r:optimal} shows the necessary conditions $K^\star|_{[1,+\infty)}\geq 1$, which given the monotonicity of $K^\star$ is equivalent to the information \textbf{$K^\star(1)=1$}. Remark \ref{r:optimal2} shows that $K^\star(M)\geq (M+1)/(3-M)$ for $M\in [1,3)$, in particular \textbf{$K^\star(2)=3$}.
\printbibliography
\end{document}